\documentclass[11pt,a4paper]{amsart}
\pdfoutput=1

  \pdfpageattr{/Group <</S /Transparency /I true /CS /DeviceRGB>>}
\usepackage[T1]{fontenc}
\usepackage[utf8]{inputenc}
\usepackage[english]{babel}
\usepackage{amsfonts,amsmath,amssymb}
\usepackage{mathtools}
\usepackage{tikz}
\usepackage{xspace}
\usepackage[colorlinks,linkcolor=links,citecolor=cites]{hyperref}
\usepackage{aliascnt}
\usepackage{varioref}
\usepackage{nameref}
\usepackage{lmodern}
\usepackage{microtype}
\usepackage{booktabs,tabularx}
\usepackage{listings}
\usepackage{enumitem}
\usepackage[labelfont=rm, subrefformat=parens]{subcaption}
\usepackage{xspace}

\setlist[enumerate, 1]{wide} \setlist[enumerate, 2]{wide = 2\parindent, leftmargin = \parindent} \setlist[enumerate, 3]{wide = 2\parindent, leftmargin = \parindent}
\setlist[description]{wide=0pt, itemsep=1ex}
\renewcommand{\descriptionlabel}[1]{\hspace{\labelsep}\textit{#1:}} 

\DeclareRobustCommand{\SkipTocEntry}[5]{}

\microtypesetup{tracking,kerning,spacing}
\microtypecontext{spacing=nonfrench}

\usetikzlibrary{calc, intersections, graphs}

\definecolor{links}{rgb}{.2,.1,.5}
\definecolor{cites}{rgb}{.5,.1,.2}

\newlength\figureheight
\newlength\figurewidth

\newcommand\RR{{\mathbb R}}

\newcommand\cT{{\mathcal T}}
\newcommand\cB{{\mathcal B}}

\newcommand{\1}{\mathbf 1}

\newcommand{\SetOf}[2]{\left\{#1\vphantom{#2}\,\right.\left|\,\vphantom{#1}#2\right\}}

\renewcommand\Vert{\operatorname{Vert}} 

\DeclareMathOperator\relint{relint}
\DeclareMathOperator\conv{conv}
\DeclareMathOperator\aff{aff}

\DeclareMathOperator\codim{codim}

\DeclareMathOperator\st{st}
\DeclareMathOperator\lk{lk}
\DeclareMathOperator\cone{cone}

\DeclareMathOperator\DP{DP}
\DeclareMathOperator\cross{cross}

\newcommand{\wrt}{with respect to\xspace}

\DeclareFontFamily{U}{mathb}{\hyphenchar\font45}
\DeclareFontShape{U}{mathb}{m}{n}{
      <5> <6> <7> <8> <9> <10> gen * mathb
      <10.95> mathb10 <12> <14.4> <17.28> <20.74> <24.88> mathb12
      }{}
\DeclareSymbolFont{mathb}{U}{mathb}{m}{n}
\DeclareMathSymbol{\precneq}{3}{mathb}{"AC}
\DeclareMathSymbol{\succneq}{3}{mathb}{"AD}

\newcommand{\biglink}{\Lambda}

\newcommand{\polymake}{\texttt{polymake}\xspace}
\newcommand{\topcom}{\texttt{TOPCOM}\xspace}

\addto\extrasenglish{%
}

\newcommand{\theoremname}{Theorem}
\newcommand{\corollaryname}{Corollary}
\newcommand{\lemmaname}{Lemma}
\newcommand{\propositionname}{Proposition}
\newcommand{\conjecturename}{Conjecture}
\newcommand{\remarkname}{Remark}
\newcommand{\examplename}{Example}
\newcommand{\definitionname}{Definition}
\newcommand{\questionname}{Question}
\newcommand{\answername}{Answer}
\newcommand{\conditionname}{Condition}
\newcommand{\importantname}{Important}
\newcommand{\observationname}{Observation}
\newcommand{\constructionname}{Construction}

\theoremstyle{plain}
    \newtheorem{theorem}{\theoremname}[section]
    \newtheorem*{theorem*}{\theoremname}
  \newaliascnt{corollary}{theorem}
    \newtheorem{corollary}[corollary]{\corollaryname}
  \newaliascnt{lemma}{theorem}
    \newtheorem{lemma}[lemma]{\lemmaname}
  \newaliascnt{proposition}{theorem}
    \newtheorem{proposition}[proposition]{\propositionname}
\theoremstyle{definition}
  \newaliascnt{conjecture}{theorem}
    \newtheorem{conjecture}[conjecture]{\conjecturename}
  \newaliascnt{remark}{theorem}
    \newtheorem{remark}[remark]{\remarkname}
  \newaliascnt{example}{theorem}
    \newtheorem{example}[example]{\examplename}
  \newaliascnt{definition}{theorem}
    \newtheorem{definition}[definition]{\definitionname}
  \newaliascnt{question}{theorem}
    
  \newaliascnt{answer}{theorem}
    
  \newaliascnt{condition}{theorem}

  \newaliascnt{observation}{theorem}
    \newtheorem{observation}[observation]{\observationname}
  \newaliascnt{construction}{theorem}
    \newtheorem{construction}[construction]{\constructionname}

\aliascntresetthe{corollary}
\aliascntresetthe{lemma}
\aliascntresetthe{proposition}
\aliascntresetthe{conjecture}
\aliascntresetthe{remark}
\aliascntresetthe{example}
\aliascntresetthe{definition}
\aliascntresetthe{question}
\aliascntresetthe{answer}
\aliascntresetthe{condition}
\aliascntresetthe{observation}
\aliascntresetthe{construction}

\makeatletter
\let\orgdescriptionlabel\descriptionlabel
\renewcommand*{\descriptionlabel}[1]{%
  \let\orglabel\label
  \let\label\@gobble
  \phantomsection
  \edef\@currentlabel{#1}%
  \let\label\orglabel
  \orgdescriptionlabel{#1}%
}
\makeatother

\newcommand{\equationref}[1]{\hyperref[#1]{Equation~\eqref{#1}}}

\numberwithin{equation}{section}

\usepackage[colorinlistoftodos, bordercolor=orange, backgroundcolor=orange!20, linecolor=orange, textsize=scriptsize
]{todonotes}

\title
[Webs of stars or how to triangulate free sums]
{Webs of stars or how to triangulate\\ free sums of point configurations}

\author{Benjamin Assarf \and Michael Joswig \and Julian Pfeifle}
\address{Benjamin Assarf, Michael Joswig, 
TU-Berlin, Str. des 17. Juni 136, D-10623 Berlin, Germany}
\email{\{assarf,joswig\}@math.tu-berlin.de}

\thanks{M.~Joswig is partially supported by Einstein Foundation Berlin and Deutsche Forschungsgemeinschaft (DFG) within the Priority Program 1489 and the Collaborate Research Centers SFB/TRR 109, SFB/TRR 195.}

\address{Julian Pfeifle, 
Universitat Polit\`ecnica de Catalunya, Dept. Matem\`atica Aplicada,
Jordi Girona 1-3, E-08034 Barcelona, Spain}
\email{julian.pfeifle@upc.edu}
\thanks{J.~Pfeifle is partially supported by projects MTM2017-83750-P, MTM2015-63791-R and 2017SGR1640}

\subjclass[2010]{52B11 (57Q15, 52B20)}
\keywords{triangulations of point configurations, polytope free sums, smooth Fano polytopes}

\begin{document}

\begin{abstract}
  The triangulations of point configurations which decompose as a free sum are classified in terms of the
  triangulations of the summands.  The methods employ two new partially ordered sets associated with any
  triangulation of a point set with one marked point, the \emph{web of stars} and the \emph{stabbing poset}.
  Triangulations of smooth Fano polytopes are discussed as a case study.
\end{abstract}

\maketitle
\tableofcontents

\section{Introduction}
\phantomsection\label{sec:intro}
\noindent
The investigation of triangulations of point configurations and their secondary fans is motivated by numerous applications in
many areas of mathematics.  For an overview see the introductory chapter of the monograph \cite{de2010triangulations} by
De Loera, Rambau and Santos.  The \emph{secondary fan} is a complete polyhedral fan which encodes the set of all
(regular) subdivisions of a fixed point configuration, partially ordered by refinement.  As secondary fans form a very
rich concept, general structural results are hard to obtain.  There are rather few infinite families of point
configurations known for which the entire set of all triangulations can be described in an explicit way; see
\cite{HerrmannJoswig:2010} for a classification which covers very many of the cases known up to now.  The purpose of
the present paper is to examine the triangulations of point configurations which decompose as a free sum, and we give a full
classification in terms of the triangulations of the summands.  A case study on a configuration of $17$ points in
$\RR^6$ underlines that, for point configurations which decompose, our methods significantly extend the range where
explicit computations are possible.

Let $P \subseteq \RR^d$ and $Q \subseteq \RR^e$ be two finite point configurations containing the origin in their
respective interiors. Their \emph{(free) sum} is the point set
\begin{equation}
   P\oplus Q \ := \
   \big(P\times \{0\}\big)
   \cup
   \big(\{0\} \times Q\big)
   \ \subseteq \ \RR^{d+e} \enspace ,
\end{equation}
and their \emph{(affine) join} is
\begin{equation}\label{eq:join}
   P * Q \ := \
   \big(\{0\} \times P\times \{0\} \big)
   \cup
   \big(\{1\} \times \{0\} \times Q\big)
   \ \subseteq \ \RR^{1+d+e} \enspace .
\end{equation}
Starting from a triangulation of the sum, the join or the Cartesian product of two point configurations, a natural question to ask is whether the triangulation can be expressed or constructed using individual triangulations of $P$ and~$Q$.
In \cite{de2010triangulations} there are several results on affine joins and Cartesian products, but none for sums.
A complete characterization for the affine joins is given by \cite[Theorem 4.2.7]{de2010triangulations}.
It turns out that every subdivision of an affine join is determined by the subdivisions of the factors in a unique way.
For the product the situation is much more involved.
Any two subdivisions of two point sets give rise to a subdivision of the product point configuration \cite[Definition 4.2.13]{de2010triangulations}, but not every subdivision of the product arises in this way.

The free sum is dual to the product, and hence it is a very natural construction to look at.
We examine how an arbitrary triangulation $\Delta_P$ of $P$ and an arbitrary triangulation $\Delta_Q$ of $Q$ give rise to a triangulation of $P\oplus Q$.
However, in order to make the construction work additional data is required.
This leads us to define \emph{webs of stars} in $\Delta_P$ and $\Delta_Q$.
These are families of star-shaped balls (containing the origin) in $\Delta_P$ and $\Delta_Q$, respectively, which satisfy certain compatibility conditions, expressed in terms of visibility from the origin.
Our first main result (\autoref{thm:sum-triangulation}) says that each triangulation of $P\oplus Q$ arises in this way.
We found it surprisingly difficult to show, however, that the resulting conditions on the summands always suffice to construct a triangulation.
This is our second main result (\autoref{thm:webs-yield-triangulation}), and completes our characterization.

One good reason for considering free sums (and their subdivisions) is that interesting classes of polytopes are closed
with respect to this construction.  This includes the \emph{smooth Fano polytopes}, which are (necessarily simplicial)
lattice polytopes with the origin as an interior lattice point such that the vertices on each facet form a lattice
basis.  For each dimension there are only finitely many smooth Fano polytopes, up to unimodular equivalence.  They are
classified in dimensions up to nine; cf.~\cite{Batyrev2007}, \cite{KN5}, \cite{OebroPhD}, and \cite{Paffenholz:1711.02936}.
Smooth Fano polytopes play a role in algebraic geometry and mathematical physics.  Interestingly, many of these
polytopes decompose as a free sum.  There is a more precise general statement conjectured
\cite[Conjecture~9]{AssarfJoswigPaffenholz:2014}, which has partially been confirmed
\cite[Theorem~1]{AssarfNill:1409.7303}.  In Section~\ref{sec:fano} we report on a case study where we apply our methods
to a six-dimensional smooth Fano polytope with $16$ vertices, which decomposes into a $2$-dimensional and a
$4$-dimensional summand.  With standard techniques (cf.~\cite{PfeifleRambau:2003} and \cite{TOPCOM}) it seems to be out
of reach to compute all its triangulations up to symmetry on a standard desktop computer within several weeks.  Yet, our
approach solves this problem on the same hardware within ten days.

Our paper is organized as follows. \autoref{sec:toolbox} starts out with investigating two partially ordered sets
which can be associated to any triangulation~$\Delta$ of a point set, in which one point is marked.  Throughout the
marked point will be the origin.  The first poset orders, with respect to inclusion, the triangulated balls of maximal dimension which contain
the origin (not necessarily as a vertex), and which are strictly star-shaped with respect to the origin.
The second poset comprises the facets of the triangulation with the partial ordering induced by visibility from the origin; this
is the \emph{stabbing poset}.

In \autoref{sec:structure} we start to investigate triangulations of sums of point
configurations.  A key step is to analyze how a triangulation $\triangle_{P\oplus Q}$ induces triangulations on the two
summands.  Here we obtain a unified treatment which simultaneously covers the case where the origin is a vertex of
$\Delta_{P\oplus Q}$ and the case where it is not.
Specifically, we show that the \emph{link map} $\sigma\mapsto\bar\biglink(\sigma)$, which assigns to each simplex $\sigma\in\triangle_{P\oplus Q}$ the cone over the origin of the link of~$\sigma$ in~$\triangle_{P\oplus Q}$, satisfies several technical properties, the most important of which is to preserve the order of the two aforementioned partial orders.

\autoref{sec:web_of_stars} generalizes the link map to \emph{webs of stars} and \emph{sum-triangulations}, and culminates in \autoref{thm:sum-triangulation}, which says that every triangulation of the free sum of two point configurations arises as a sum-triangulation.

Finally, \autoref{sec:sum-triangs_are_triangs} is dedicated to the proof of the converse
direction, namely that \emph{every} pair of triangulations of the summands can be used to construct a triangulation of
the sum.  However, the correspondence is not one-to-one, meaning that different pairs of triangulations of the summands
may produce the same sum-triangulation.

In order to show the applicability and usefulness of our methods, in
\autoref{sec:fano} we analyze one specific point configuration in detail: the free sum $\DP(2)\oplus\DP(4)$ of two
\emph{del Pezzo polytopes}, of dimensions two and four.  This is a smooth Fano polytope in dimension six with $16$
vertices; including the origin gives a total of $17$ points.  Using the triangulations of the summands as
input (obtained via \topcom \cite{TOPCOM}) we compute all triangulations of $\DP(2)\oplus\DP(4)$ with
\polymake~\cite{DMV:polymake}.

We close the paper with a conjecture about regular triangulations and an appendix on an algorithmic detail.

\section{Toolbox}
\phantomsection\label{sec:toolbox}

\subsection{Simplices in direct sums}
We start out with some relevant basic facts about triangulations of a finite point set
$P \subset \RR^d$.
An \emph{interior point} of $P$ is a point in $P$ which is contained in the interior of the convex hull $\conv P$.
Clearly, the convex hull of $P$ needs to be full-dimensional in order to have any interior points.
Now let $Q\subset \RR^e$ be another configuration of finitely many points.
Throughout the paper, we will assume that the origin $0$ (in $\RR^d$ and $\RR^e$, respectively) is an interior point of both~$P$ and~$Q$.
This entails that $P$ linearly spans the entire space $\RR^d$, and $Q$ spans~$\RR^e$ likewise.
The origin in $\RR^{d+e}$ plays a special role in $P\oplus Q$, since it is the only point in the intersection $(P\times \{ 0 \}) \cap (\{ 0 \}\times Q)$.

We denote triangulations of $P$, $Q$ and~$P\oplus Q$ by $\triangle_P$, $\triangle_Q$ and~$\triangle_{P\oplus Q}$,
respectively.  As usual, a simplex~$\sigma$ of a triangulation $\triangle$ is the convex hull of its vertices, and its
dimension is the dimension of their affine hull.  We write~$\triangle^{=k}$ for the set of all simplices of dimension~$k$
in~$\triangle$, and $\partial \triangle$ for the boundary complex of~$\triangle$.

Consider a full-dimensional simplex $\sigma \in \triangle_{P\oplus Q}^{=d+e}$.
Because the vertex set of~$\sigma$ is affinely independent, it contains at most $d+1$ points of $P$ and at most $e+1$ points of $Q$.
On the other hand, since $\sigma$ is a $(d+e)$-simplex, it has exactly $d+e+1$~vertices.
Therefore, $\sigma$~contains at least $d$~points of $P$ and at least $e$~points of $Q$, and we express $\sigma$ as $\sigma=\conv(\sigma_P, \sigma_Q)$ with
\begin{equation}\label{def:sigma_P}
  \begin{aligned}
    \sigma_P \ &:= \ \conv\big(\Vert\sigma \cap (P \times\{ 0 \})\big) \quad \text{and}
    \\
    \sigma_Q \ &:= \ \conv\big(\Vert\sigma \cap (\{ 0 \}\times Q)\big) \enspace ,
  \end{aligned}
\end{equation}
where $\Vert\sigma$ denotes the set of vertices of $\sigma$.
\begin{observation}
  \label{obs:full-dim}
  If $0\notin\Vert\sigma$, then exactly one of the simplices $\sigma_P$, $\sigma_Q$~is full-dimensional, and the other
  has codimension~$1$ in the affine span of its containing polytope.  On the other hand, if $0\in\Vert\sigma$ then
  both simplices are full-dimensional.
\end{observation}

We will be a bit imprecise with our notation.
Often we will confuse~$P$ with $P\times \{0\}$, and $\sigma_P$~with its canonical projection to the linear subspace~$\RR^d$.
Accordingly, instead of $\sigma=\conv(\sigma_P, \sigma_Q)$ we will also write $\sigma=\sigma_P\oplus \sigma_Q$.

Collecting all simplices of $\triangle_{P\oplus Q}$ that lie in $\RR^d\times \{0\}$ or $\{ 0 \}\times \RR^e$ yields simplicial complexes on the vertex sets of $P$ and $Q$ that do not necessarily cover the respective convex hulls.
In \autoref{sec:structure} we prove that these complexes can be extended to proper triangulations.
Hence there exist triangulations $\triangle_P$ of $P$ and $\triangle_Q$ of~$Q$ such that every full-dimensional cell of $P\oplus Q$ is the sum of two cells of those two triangulations.
We defer the obvious question of how those cells are to be combined into a decomposition of $P \oplus Q$ until we describe our main construction, the \emph{sum-triangulation}, in \autoref{def:sum_triang}.
From any two fixed triangulations of the summands, it can produce several triangulations of the sum.
Conversely, every triangulation of the sum of two polytopes is a sum triangulation and can be produced from triangulations of the summands, but not necessarily in a unique way (\autoref{cons:refine-triangulation}).

\subsection{Stars, links, and two new posets}
Let $\triangle$ be a triangulation of a point configuration in~$\RR^d$ and $\sigma$ be a face in~$\triangle$.
\begin{definition}[star/link/restriction]
  The \emph{(closed) star} $\st_\triangle(\sigma)$ of $\sigma$ is the subcomplex of~$\triangle$
  consisting of all simplices containing~$\sigma$, and all their faces.  The \emph{link} of $\sigma$ is the
  simplicial complex $\lk_\triangle(\sigma) := \{ \tau \in \st_\triangle(\sigma) \,\mid\, \sigma\cap \tau = \emptyset
  \}$.

  Consider a point $x$ in the set covered by $\triangle$, which, however, does not need to be a vertex, and let $\sigma$ be the minimal face containing it.
  We let
  \[
  \st_\triangle(x) \:=\ \st_\triangle(\sigma) \quad\text{and}\quad \lk_\triangle(x) \:=\ \lk_\triangle(\sigma) \enspace .
  \]

  For any closed set $S \subseteq \RR^d$ we call $\triangle|_S := \{ \sigma \in \triangle \;\mid\; \sigma \subseteq S \}$ the \emph{restriction} of $\triangle$ to $S$.
\end{definition}

Traditionally, a set $S$ is called \emph{star shaped} with respect to the point $x\in S$ if for every $y\in S$ the line segment $\overline{xy}$ is completely contained in $S$.
We need a slightly stricter version of this generalization of convexity.
\begin{definition}
  A set $S$ is \emph{strictly star shaped} with respect to $x\in S$ if for every $y\in S$ the line segment $\overline{xy}$ is completely contained in $\relint(S)\cup\{y\}$.
\end{definition}
Thus, the point $x$ must be contained in the relative interior of $S$, and the line segment $\overline{xy}$ is only allowed to intersect the boundary of $S$ in $y$.
Another way of saying the same is that every ray starting at $x$ can intersect $\partial S$ in at most one point.

\begin{lemma}\label{lemma:strictly_star_shaped_star}
  Let $\triangle$ be a triangulation of a point configuration, and let $x$ be a point in its (relative) interior. Then $\st_\triangle(x)$ is strictly star shaped with respect to~$x$.
\end{lemma}
\begin{proof}
  Because every simplex is convex, $\st_\triangle(x)$ is star shaped with respect to~$x$.
  The relative interior of the star is $\st_\triangle(x) \setminus \lk_\triangle(x)$, and it is clear that it contains~$x$.
  Let $F$ be a maximal cell in the boundary of $\st_\triangle(x)$.
  Then $F$ is contained in $\lk_\triangle(x)$, and it does not contain~$x$.
  It follows that the vertices of $F$ and $x$ form an affinely independent set, and
  thus the intersection of $F$ and any line segment $\overline{xy}$ for $y \in F$ is just the point~$y$.
\end{proof}

\begin{observation}\label{obs:pl-ball}
  As a direct consequence of Lemma~\ref{lemma:strictly_star_shaped_star}, the link of any interior cell forms a triangulated sphere.
  Notice also that the link of any boundary cell is a triangulated ball; cf.~\cite[Chapter~1]{Hudson.pl}.
\end{observation}

Two partially ordered sets will play a crucial role in the rest of the paper. The first poset is associated to any triangulation $\triangle$ of a point configuration in $\RR^d$, namely
\[
   \cB_\triangle^k(x)
   \ = \
   \left\{\;
     \parbox{7cm}{subcomplexes of $\triangle$ that are $k$-dimensional balls and strictly star shaped with respect to $x$}
    \;\right\}
  \cup\{\emptyset\},
\]
partially ordered by inclusion. For convenience, we abbreviate $B_{\triangle} := B_\triangle^d(0)$.

The second partial order is defined on the simplices in $\triangle$ of the same dimension. We say that $\sigma$ \emph{precedes~$\tau$ in the stabbing order}, and write $\sigma \preceq \tau$,
if the dimensions of $\sigma$ and $\tau$ are equal and $\sigma=\tau$ or (compare~\autoref{fig:stabbing_ray})
\begin{itemize}
  \item for every linear or affine hyperplane that separates $\sigma$~and~$\tau$ (not necessarily strictly), $\sigma$ lies in the same closed half space as~$0$; and
  \item $\sigma$ and $\tau$ are separated by at least one strictly affine hyperplane~$H$.
\end{itemize}
The minimal elements in the $\preceq$-ordering are the simplices in $\st_{\triangle}(0)$, as they already contain the origin.
Throughout, we write $\sigma\prec\tau$ if $\sigma\preceq\tau$ and $\sigma\neq\tau$.

\begin{figure}[htb]
    \centering
    \subcaptionbox{\label{fig:stabbing_ray:a}
      If $\sigma\preceq \tau$, we can find a stabbing ray spanned by a point~$r$ on an affine separating hyperplane.
    }{
      \begin{tikzpicture}[scale=1]
        \draw[opacity=0] (-0.3,2.2) -- (4.3,-1.2);

        \coordinate (sigma) at (1,0.5);
        \coordinate (tau) at (3,0);
        \coordinate (o) at (0,0);
        \coordinate (ray) at (4,0.5);

        \fill[blue!20] (sigma) ellipse (0.3 and 0.5);
        \node at ($(sigma)+(0,0.15)$) {\footnotesize $\sigma$};

        \fill[red!20] (tau) ellipse (0.3 and 0.5);
        \node at ($(tau)+(0,-0.15)$) {\footnotesize $\tau$};

        \draw[black, -> ,name path=line 1] (o) -- (ray);
        \draw[orange, ultra thick, name path=line 2] (3,2) -- (1,-1);

        \fill (o) circle (1.5pt);
        \node at ($(o)+(-0.3,0)$) {\footnotesize $0$};

        \fill[name intersections={of=line 1 and line 2, by=x}] (x) circle (1.5pt)node[above] {\footnotesize $r$};

        \fill ($(o)!(sigma)!(ray)$) circle (1.5pt);
        \fill ($(o)!(tau)!(ray)$) circle (1.5pt);
        \node at ($(o)!(sigma)!(ray)+(0,-0.3)$) {\footnotesize $\lambda r$};
        \node at ($(o)!(tau)!(ray)+(0,0.3)$) {\footnotesize $\mu r$};
      \end{tikzpicture}
    }
    \hspace{.1\linewidth}
    \subcaptionbox{\label{fig:stabbing_ray:b}
      If $\sigma$ and $\tau$ are not comparable we can find a linear separating hyperplane.
    }{
      \begin{tikzpicture}[scale=1]
        \draw[opacity=0] (-0.8,2.2) -- (3.8,-1.2);

        \coordinate (sigma) at (1,1);
        \coordinate (tau) at (2,-0.3);
        \coordinate (o) at (0,0);
        \coordinate (ray) at (3.5,1);

        \fill[blue!20] (sigma) ellipse (0.3 and 0.5);
        \node at ($(sigma)+(0,0)$) {\footnotesize $\sigma$};

        \fill[red!20] (tau) ellipse (0.3 and 0.5);
        \node at ($(tau)+(0,0)$) {\footnotesize $\tau$};

        \draw[orange, ultra thick, shorten <=-.7cm] (o) -- (ray);

        \draw[blue!50, ->, shorten >=-.7cm] (o) -- ($(sigma)+(-0.13,0.5)$);
        \draw[blue!50, ->, shorten >=-.7cm] (o) -- ($(sigma)+(0.04,-0.5)$);

        \draw[red!70!black, ->, shorten >=-.7cm] (o) -- ($(tau)+(0,0.5)$);
        \draw[red!70!black, ->, shorten >=-.7cm] (o) -- ($(tau)+(-0.04,-0.5)$);

        \fill (o) circle (1.5pt);
        \node at ($(o)+(-0.3,0.2)$) {\footnotesize $0$};
      \end{tikzpicture}
    }
    \caption{\label{fig:stabbing_ray}
      Illustrating the partial order~$\preceq$. In
      \subref{fig:stabbing_ray:a} $\sigma$ precedes $\tau$, in \subref{fig:stabbing_ray:b} $\sigma$ and $\tau$ are not comparable.}
  \end{figure}
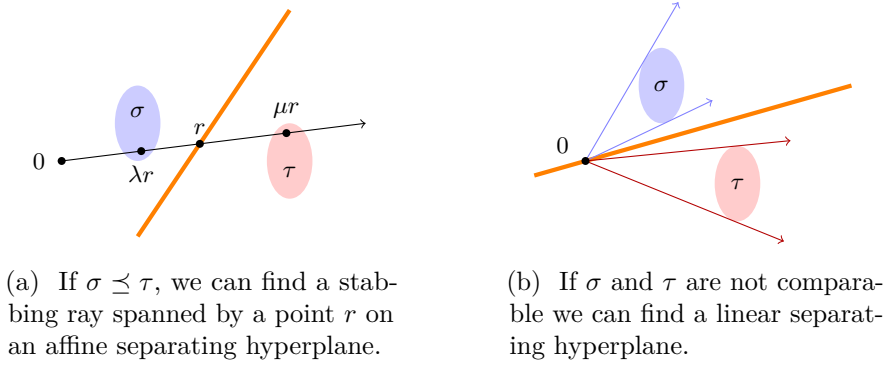
\begin{lemma}\label{lemma:find_intersecting_ray}
  For any two distinct, comparable simplices $\sigma \prec \tau$ in a triangulation~$\triangle$ in $\RR^d$, there exists a ray from the origin that stabs first~$\sigma$ and then~$\tau$, i.e., there exists an $r\in \RR^d\setminus\{0\}$ and $0\le\lambda < \mu$ such that $\lambda r\in \sigma$ and $\mu r\in \tau$.
\end{lemma}
\begin{proof}
  First we briefly discuss that if $\sigma\prec\tau$ and both are stabbed by a ray~$\rho$ through the origin, we may assume that two intersection points are of the form $\lambda r\in \sigma$ and $\mu r\in \tau$ with  $r\in\RR\setminus\{0\}$ and $0\le\lambda<\mu$.
  To get some intuition for this, consider the separating affine hyperplane~$H$ whose existence is guaranteed by the definition of~$\preceq$, and choose $r:=\rho\cap H$.
  Then $\lambda \le \mu$ is clear because $\sigma$~is contained in the same half-space of~$H$ as the origin, by the definition of~$\preceq$, but is separated from~$\tau$ by~$H$; see \autoref{fig:stabbing_ray:a}.
  Furthermore, a ray with $\lambda < \mu$ exists even if the intersection $\sigma \cap \tau$ is just one point or empty.

  To formally prove the lemma, we assume that there exists no stabbing ray~$\rho$, and construct a linear hyperplane~$L$ that separates~$\sigma$ from~$\tau$, and that can be perturbed to leave the origin on either side while still separating $\sigma$ from~$\tau$. The existence of such an~$L$ then directly contradicts $\sigma\preceq\tau$.

  To construct $L$, we distinguish two cases.
  First, suppose that $\sigma \cap \tau =\emptyset$. As there exists no stabbing ray for $\sigma$~and~$\tau$, the two polyhedral cones induced by those two cells, $\cone(\sigma):=\SetOf{\lambda x}{x\in\sigma, \lambda \ge 0}$ and $\cone(\tau)$, intersect only in the origin.
  Therefore, we can separate these cones by a linear hyperplane~$L$, which of course also separates $\sigma$ and $\tau$.
  It is clear that $L$~can be perturbed or moved slightly in the required way; see \autoref{fig:stabbing_ray:b}.

  Finally, assume that $\sigma\cap\tau \ne \emptyset$, and note that $\aff(\sigma\cap\tau)$ is contained in every hyperplane that separates $\sigma$~and~$\tau$. If all such hyperplanes are linear, we have our desired contradiction to $\sigma\preceq\tau$, so we may assume that there exists a strictly affine separating hyperplane~$H$, so that $0\notin H$.  Then we conclude $0\notin\aff(\sigma\cap\tau)$ because $\aff(\sigma\cap\tau)\subset H$.
  Our assumption about the non-existence of a stabbing ray~$\rho$ means that the cone over $\aff(\sigma\cap\tau)$ only intersects $\sigma$ or $\tau$ in $\sigma\cap\tau$.
  Thus there exists a separating linear hyperplane~$L$ with $\sigma\cap \tau = \sigma \cap L = L\cap \tau$ and $\codim\aff(\sigma\cap\tau)\ge1$.
  One can start with any linear hyperplane containing $\aff(\sigma\cap\tau)$ and perturb it suitably to obtain~$L$.
\end{proof}

\autoref{lemma:find_intersecting_ray} yields a necessary condition for distinct simplices to be $\prec$-comparable.
However, from an algorithmic point of view, the definition of $\sigma \prec \tau$ is inconvenient, because one has to consider \emph{all} separating hyperplanes.
For this reason, \autoref{fig:appendix:decision_tree} in the Appendix illustrates an algorithm for deciding whether $\sigma \prec \tau$ or not.

\section{Links in triangulations of free sums}
\phantomsection\label{sec:structure}
\noindent
Let $\triangle_{P\oplus Q}$ be a triangulation of the free sum of two point configurations $P$ and $Q$.
We use~\eqref{def:sigma_P} to express each simplex $\sigma\in\triangle_{P\oplus Q}$ as $\sigma = \sigma_P \oplus \sigma_Q$.
In case $\sigma\subset P$, respectively $\sigma\subset Q$, we take $\sigma_Q=\emptyset$, respectively $\sigma_P=\emptyset$.
\autoref{obs:full-dim} says what is known about the dimensions of $\sigma_P$ and $\sigma_Q$.
From now on, let $\sigma = \sigma_P \oplus \sigma_Q$ and $\tau=\tau_P \oplus \tau_Q$ be two cells of $\triangle_{P\oplus Q}$.

\subsection{The link map in a triangulation of a sum}
Throughout the rest of the paper, we abbreviate
\begin{align}
   \Lambda(\sigma)
   \ & := \
   \lk_{\triangle_{P\oplus Q}}(\sigma) \quad \text{and}
   \\
   \bar\biglink(\sigma) \label{eq:barbiglink}
   \ &:= \
   0\star\biglink(\sigma)
   \ = \
   \SetOf{\lambda x}{x\in\biglink(\sigma),\; \lambda\in[0,1]} \enspace.
\end{align}
Here ``$\star$'' is the join operation.
Since $0$ is a single point, the set $\bar\biglink(\sigma)$ is the cone over the link $\biglink(\sigma)$.
We write ``$\star$'' instead of ``$*$'' since, in contrast with the general situation in \eqref{eq:join} this cone has a natural realization in $\RR^{d+e}$, so that we do not need to increase the dimension.
For now, the \emph{link map}
\[
  \sigma\mapsto\bar\biglink(\sigma)
\]
is defined for all simplices $\sigma\in\triangle_{P\oplus Q}$, but we will adjust this domain of definition in a moment.

The subsequent sections prove the following structural properties of~$\bar\biglink$:
\begin{description}
  \item [Domain] The triangulation $\triangle_{P\oplus Q}$ induces (not necessarily unique) triangulations $\triangle_P$ of~$P$ and $\triangle_Q$ of~$Q$, and we take $\triangle_P\cup\triangle_Q$ as the domain of definition of $\bar\biglink$ (Construction~\ref{cons:refine-triangulation}).
  \item [Range] For a full-dimensional simplex  $\sigma$ in the domain of~$\bar\biglink$, the restriction $\triangle_{P\oplus Q}|_{\bar\biglink(\sigma)}$ is a strictly star-shaped ball (Proposition~\ref{prop:strictly_starshaped}).
  \item [Preserves order] For $d$-dimensional $\sigma_P\preceq\tau_P\in\triangle_P$, we obtain
      \[
        \bar\biglink(\sigma_P) \ \subseteq \ \bar\biglink(\tau_P) \enspace;
      \]
      the analogous claim holds for $e$-dimensional $\sigma_Q\preceq\tau_Q\in\triangle_Q$ (Proposition~\ref{prop:order_preserving}).
  \item [Complementarity] For $d$-dimensional $\sigma_P$ and $e$-dimensional $\tau_Q$ we have
      \[
        \sigma_P \subseteq \bar\biglink(\tau_Q)
          \iff
        \tau_Q \not\subseteq \bar\biglink(\sigma_P) \enspace.
      \]
      This is the content of Proposition~\ref{prop:compatible}.
\end{description}

\subsection{The domain: constructing triangulations of the summands}
\label{sec:constructing}
As a first step we define the simplicial complexes
\begin{align*}
  \tilde\triangle_P \
  &=
    \ \SetOf{\sigma_P}{\sigma = \sigma_P \oplus \sigma_Q \in \triangle_{P\oplus Q}} \quad \text{and} \quad
  \\
  \tilde\triangle_Q \
  &=
    \ \SetOf{\sigma_Q}{\sigma = \sigma_P \oplus \sigma_Q \in \triangle_{P\oplus Q}} \enspace .
\end{align*}

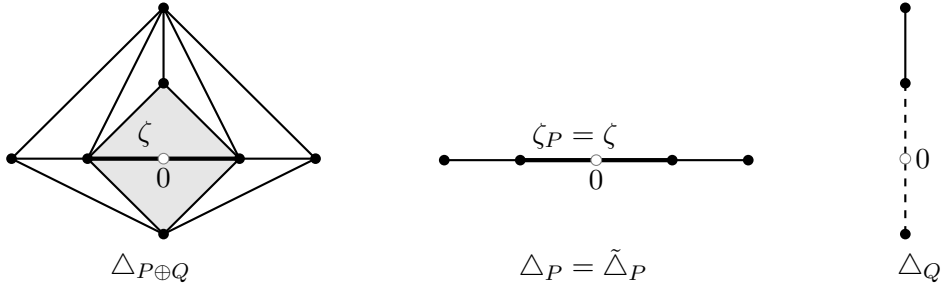
\begin{figure}[htb]
  \begin{minipage}{.3\textwidth}\centering
    \begin{tikzpicture}[scale=1]
      \tikzstyle{edge} = [draw,thick,-,black]

      \coordinate (v0) at (-2,0);
      \coordinate (v1) at (-1,0);
      \coordinate (v2) at (0,0);
      \coordinate (v3) at (1,0);
      \coordinate (v4) at (2,0);

      \coordinate (w0) at (0,-1);
      \coordinate (w1) at (0,0);
      \coordinate (w2) at (0,1);
      \coordinate (w3) at (0,2);

      \fill[gray!20] (v1) -- (w0) -- (v3) -- (w2) -- cycle;
      
      \draw[edge] (v0) -- (w0) -- (v4) -- (w3) -- cycle;
      \draw[edge] (v1) -- (w0) -- (v3) -- (w3) -- cycle;
      \draw[edge] (v1) -- (w2) -- (v3);
      \draw[edge] (w2) -- (w3);
      \draw[edge] (v0) -- (v4);
      \draw[edge,ultra thick] (v1) to node[above right, near start]{$\zeta$} (v3);

      \foreach \point in {v0,v1,v3,v4,w0,w2,w3}
      \fill[black] (\point) circle (2pt);

      \filldraw[draw=gray, fill=white] (v2) circle (2pt);  
      \node[below] at (v2) {$0$};
      
    \end{tikzpicture}
    $\triangle_{P\oplus Q}$
  \end{minipage}
  \hfill
  \begin{minipage}{.3\textwidth}\centering
    \begin{tikzpicture}[scale=1]
      \tikzstyle{edge} = [draw,thick,-,black]

      \coordinate (v0) at (-2,0);
      \coordinate (v1) at (-1,0);
      \coordinate (v2) at (0,0);
      \coordinate (v3) at (1,0);
      \coordinate (v4) at (2,0);

      \coordinate (w0) at (0,-1);
      \coordinate (w1) at (0,0);
      \coordinate (w2) at (0,1);
      \coordinate (w3) at (0,2);

      \draw[edge] (v0) -- (v4);
      \draw[edge,ultra thick] (v1) to node[above right, at start]{$\zeta_P=\zeta$} (v3);

      \foreach \point in {v0,v1,v3,v4}
      \fill[black] (\point) circle (2pt);

      \foreach \point in {w0,w2,w3}
      \fill[white] (\point) circle (2pt);

      \filldraw[draw=gray, fill=white] (v2) circle (2pt);  
      \node[below] at (v2) {$0$};
    \end{tikzpicture}
    $\triangle_{P}=\tilde\triangle_P$
  \end{minipage}
  \hfill
  \begin{minipage}{.1\textwidth}\centering
    \begin{tikzpicture}[scale=1]
      \tikzstyle{edge} = [draw,thick,-,black]

      \coordinate (v0) at (-2,0);
      \coordinate (v1) at (-1,0);
      \coordinate (v2) at (0,0);
      \coordinate (v3) at (1,0);
      \coordinate (v4) at (2,0);

      \coordinate (w0) at (0,-1);
      \coordinate (w1) at (0,0);
      \coordinate (w2) at (0,1);
      \coordinate (w3) at (0,2);

      \draw[edge] (w2) -- (w3);
      \draw[edge, dashed] (w0) -- (w2);

      \foreach \point in {w0,w1,w2,w3}
      \fill[black] (\point) circle (2pt);

      \filldraw[draw=gray, fill=white] (v2) circle (2pt);  

      \node[right] at (w1) {$0$};
    \end{tikzpicture}

    $\triangle_{Q}$
  \end{minipage}
  \caption{Constructing $\triangle_P$ and $\triangle_Q$ from $\triangle_{P\oplus Q}$.
    The star-shaped ball $\bar\biglink(\zeta)$ in $\triangle_{P\oplus Q}$ (left) is shaded.
    The dashed cells form the filling of $\triangle_Q$ (right).}
  \label{fig:construction}
\end{figure}

\begin{construction}\label{cons:refine-triangulation}
  Let $\zeta$ be the unique cell of $\triangle_{P\oplus Q}$ which contains the origin in its relative interior.
  We consider two cases.
  If $\zeta$ is a vertex, then $\tilde\triangle_P$ and $\tilde\triangle_Q$ cover $P$ and $Q$, respectively.
  In this case we can take $\triangle_P=\tilde\triangle_P$ and $\triangle_Q=\tilde\triangle_Q$.

  Otherwise, we have $\dim\zeta\geq 1$.
  We decompose $\zeta=\zeta_P\oplus\zeta_Q$ into its $P$- and its $Q$-part.
  Now the origin either lies in $\relint\zeta_P$ or in $\relint\zeta_Q$.
  By symmetry we may assume that the origin lies in the $P$-part.
  It follows that the same holds for all cells of $\triangle_{P\oplus Q}$ containing $\zeta$.
  These are precisely the cells containing the origin.
  We conclude that $\tilde \triangle_P$ covers all of $P$, and we let $\triangle_P := \tilde \triangle_P$.
  
  Since $\zeta$ is the unique minimal cell containing $0$ it follows that $\zeta_P=\zeta$ and $\zeta_Q=\emptyset$.
  What is left to define is $\triangle_Q$.
  As $\zeta$ has positive dimension the subcomplex $\tilde\triangle_Q$ does not cover $Q$: there is no cell in $\tilde\triangle_Q$ containing $0$.
  Yet the star $\bar\biglink(\zeta)$ of $0$ in $\triangle_{P\oplus Q}$ is strictly star-shaped with respect to $0$.
  It follows that $\bar\biglink(\zeta)\cap Q$ is also strictly star-shaped, and this is precisely the region which is not covered by $\tilde\triangle_Q$.
  From Observation~\ref{obs:pl-ball} we know that $\biglink(\zeta)$ is a triangulated $(e{-}1)$-sphere which forms a subcomplex of $\tilde\triangle_Q$.
  Now we add cones of all cells in $\biglink(\zeta)$ with apex $0$ to obtain the desired triangulation $\triangle_Q$ of $Q$, and this contains $\tilde\triangle_Q$ as a proper subcomplex.
  The added cones form the \emph{filling} of $\triangle_Q$; cf.\ Figure~\ref{fig:construction} for a sketch.
\end{construction}

  We extend the link map $\bar\biglink$ to simplices in the filling by setting $\bar\biglink(\tau)=\st_{\triangle_P}(0)$ for every $\tau\in\triangle_Q\setminus\tilde\triangle_Q$ if the origin lies in the $P$-part, and making the definition with $P$~and~$Q$ interchanged if the origin lies in the $Q$-part.
  The domain of definition of the link map is then $\triangle_P\cup\triangle_Q$.

  \smallskip
In the above construction the special role of the cell~$\zeta$ leads to an asymmetry between $\triangle_P$ and $\triangle_Q$.
However, we can also rewrite the triangulations of the summands as follows:
\begin{equation}\label{eq:constructing}
  \triangle_P \ = \ \SetOf{\sigma\cap\RR^d}{\sigma\in\triangle_{P\oplus Q}} \,, \quad
  \triangle_Q \ = \ \SetOf{\sigma\cap\RR^{\vphantom{d} e}}{\sigma\in\triangle_{P\oplus Q}} \,.
\end{equation}
It will turn out to be useful to have both descriptions.
Notice that if $\zeta\in\triangle_{P\oplus Q}$ contains the origin in its relative interior and $\triangle_P$ is a subcomplex of $\triangle_{P\oplus Q}$, then $0$ is necessarily a vertex of $\triangle_Q$.

\subsection{Images of the link map are strictly star shaped balls}
As before we consider the triangulation $\triangle_{P\oplus Q}$ of $P\oplus Q$ given, from which we obtain $\triangle_P$ and $\triangle_Q$ by Construction~\ref{sec:constructing} or \eqref{eq:constructing}.
For each simplex in a summand we need to examine its link in the free sum.

\smallskip
The following lemma, which is a version of Pasch's theorem, is a basic tool for the rest of this paper.
\begin{lemma}\label{lemma:two_line_crossing}
  For linearly independent $x,y \in \RR^d$ and all $\lambda,\mu > 1$, the line segment between $x$ and $\mu y$ crosses the line segment between $y$ and $\lambda x$.
\end{lemma}
\begin{proof}
  $tx + (1-t)\mu y = s y + (1-s)\lambda x$ for  $t = \frac{\lambda(\mu-1)}{\mu\lambda-1}$ and $s=(1-t)\mu$.
\end{proof}

\begin{proposition}\label{prop:strictly_starshaped}
  Let $\triangle_{P\oplus Q}\subset\RR^{d+e}$, $\triangle_P\subset\RR^d$ and $\triangle_Q\subset\RR^e$ be as above.
  Then the image $\bar\biglink(\sigma)$ of any full-dimensional simplex $\sigma\in\triangle_P\cup\triangle_Q$ under the link map is a strictly star-shaped ball with respect to the origin.
\end{proposition}

\begin{proof}
  If $\sigma\in\triangle_Q$ lies in the filling, its image $\bar\biglink(\sigma)=\st_{\triangle_P}(0)$ is a strictly star-shaped ball by Construction~\ref{cons:refine-triangulation}.

  Otherwise, we assume without loss of generality that $\sigma=\sigma_P\in\triangle_P$ has dimension~$d$, and that $d,e\ge2$.
  To show that each ray that emanates from the origin meets the $(e-1)$-dimensional triangulated sphere $\biglink(\sigma_P)\subset\{0\}\times\RR^e$ at most once,
  we assume to the contrary that the ray~$\rho\subset\RR^e$ meets~$\biglink(\sigma_P)$ at least twice, say in points $x,\lambda x$ with $\lambda>1$.
  These points may be assumed to lie in distinct faces of~$\biglink(\sigma_P)$, say $x\in\tau_1$ and $\lambda x\in\tau_2$:
  If they lie in the same face~$\tau$, the ray~$\rho$ passes through the relative interior of~$\tau$ and may be extended until it hits $\partial\tau$ in two distinct points, and these points are contained in distinct faces, adjacent to $\tau$, of the triangulated sphere $\biglink(\sigma_P)$.

  The final step now uses two points $y,\mu y\in\sigma_P$ with $\mu>1$, whose existence is guaranteed because $\sigma$~has full dimension~$d$.
  We have thus found two distinct cells $\sigma_P\oplus\tau_1$, $\sigma_P\oplus\tau_2$ of the triangulation $\triangle_{P\oplus Q}$, with $x,\mu y\in\sigma_P\oplus\tau_1$ and $\lambda x,y\in\sigma_P\oplus\tau_2$, whose relative interiors intersect by Lemma~\ref{lemma:two_line_crossing}, see Figure~\ref{fig:two_line_crossing:a}.
  This contradiction concludes the proof.
\end{proof}

  \begin{figure}[htb]
    \centering
    \begin{subfigure}[b]{.45\linewidth}\centering
      \begin{tikzpicture}
        \draw[->] (0,0) -- (3,0);
        \draw[->] (0,0) -- (0,2.5);

        \coordinate (my) at (0,1.5);
        \coordinate (y) at (0,1);

        \coordinate (x) at (1.1,0);
        \coordinate (lx) at (1.9,0);

        \fill[black!20] ($(my)+(0,-0.1)$) ellipse (0.3 and 0.6);
        \node at ($(my)+(0,0.2)$) {\footnotesize $\sigma_P$};

        \fill[blue!20] ($(x)+(-0.2,0)$) ellipse (0.4 and 0.2);
        \fill[red!20] ($(lx)+(0.2,0)$) ellipse (0.4 and 0.2);
        \node at ($(x)+(-0.3,0)$) {\footnotesize $\tau_1$};
        \node at ($(lx)+(0.3,0)$) {\footnotesize $\tau_2$};

        \draw[gray, thick, name path=line 1] (x) -- (my);
        \draw[gray, thick, name path=line 2] (lx) -- (y);

        \fill[name intersections={of=line 1 and line 2, by=z}] (z) circle (1.5pt)node[above right] {\footnotesize $z$};

        \fill (my) circle (1.5pt);
        \fill (y) circle (1.5pt);
        \node at ($(my)+(-0.5,0)$) {\footnotesize $\mu y$};
        \node at ($(y)+(-0.5,0)$) {\footnotesize $y$};

        \fill (x) circle (1.5pt);
        \fill (lx) circle (1.5pt);
        \node at ($(x)+(0,-0.3)$) {\footnotesize $x$};
        \node at ($(lx)+(0,-0.3)$) {\footnotesize $\lambda x$};
        \fill (0,0) circle (1.5pt);
        \node at (-0.3,-0.3) {\footnotesize $0$};
      \end{tikzpicture}
      \caption{\label{fig:two_line_crossing:a}
        The point $z$ lies in the relative interior of both $\sigma_P\oplus\tau_1$ and $\sigma_P\oplus\tau_2$.}
    \end{subfigure}
    \hfill
    \begin{subfigure}[b]{.45\linewidth}\centering
      \begin{tikzpicture}
        \draw[->] (0,0) -- (3,0);
        \draw[->] (0,0) -- (0,2.5);

        \coordinate (ms) at (0,1.5);
        \coordinate (s) at (0,1);

        \coordinate (r) at (1.1,0);
        \coordinate (lr) at (1.9,0);

        \fill[red!20] ($(ms)+(0,0.2)$) ellipse (0.25 and 0.4);
        \node at ($(ms)+(0,0.3)$) {\footnotesize $\tau_P$};

        \fill[blue!20] ($(s)+(0,-0.2)$) ellipse (0.25 and 0.4);
        \node at ($(s)+(0,-0.3)$) {\footnotesize $\sigma_P$};

        \fill[red!20] ($(r)+(-0.2,0)$) ellipse (0.4 and 0.2);
        \fill[blue!20] ($(lr)+(0.2,0)$) ellipse (0.4 and 0.2);
        \node at ($(r)+(-0.3,0)$) {\footnotesize $\tau_Q'$};
        \node at ($(lr)+(0.3,0)$) {\footnotesize $\sigma_Q'$};

        \draw[gray, thick, name path=line 1] (r) -- (ms);
        \draw[gray, thick, name path=line 2] (lr) -- (s);


        \fill (ms) circle (1.5pt);
        \fill (s) circle (1.5pt);
        \node at ($(ms)+(-0.5,0)$) {\footnotesize $\mu r$};
        \node at ($(s)+(-0.5,0)$) {\footnotesize $r$};

        \fill (r) circle (1.5pt);
        \fill (lr) circle (1.5pt);
        \node at ($(r)+(0,-0.3)$) {\footnotesize $\lambda x$};
        \node at ($(lr)+(0,-0.3)$) {\footnotesize $x$};
        \fill (0,0) circle (1.5pt);
        \node at (-0.3,-0.3) {\footnotesize $0$};
      \end{tikzpicture}
      \caption{\label{fig:two_line_crossing:b}
        The simplices $\sigma'=\conv(\sigma_P,\sigma_Q')$ and $\tau'=\conv(\tau_P,\tau_Q')$ intersect non-trivially in their relative interiors.}
    \end{subfigure}
    \caption{\label{fig:two_line_crossing}
      Illustration of the proofs of \autoref{prop:strictly_starshaped}, shown in~\subref{fig:two_line_crossing:a}, and \autoref{prop:order_preserving}, shown in \subref{fig:two_line_crossing:b}.}
  \end{figure}
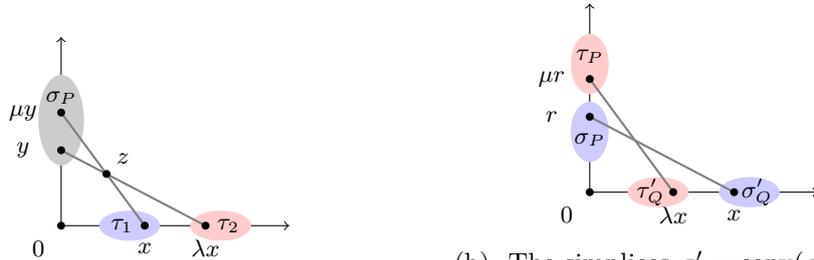

\subsection{The link map is order preserving}
\begin{proposition}\label{prop:order_preserving}
  Consider cells $\sigma = \sigma_P \oplus \sigma_Q$ and $\tau = \tau_P \oplus \tau_Q$  in $\triangle_{P\oplus Q}^{=d+e}$.
  If $\dim\sigma_P=\dim\tau_P=d$ and $\sigma_P \preceq \tau_P$, then $\bar\biglink(\sigma_P) \subseteq \bar\biglink(\tau_P)$.
  The analogous statement holds if $\dim\sigma_Q=\dim\tau_Q=e$.
\end{proposition}
\begin{proof}
  If $\sigma_P = \tau_P$ there is nothing to show.
  So we assume that $\sigma_P\prec\tau_P$, whence $0 \notin \tau_P$.
  Then \autoref{lemma:find_intersecting_ray} yields $r\in\RR^{d}\times\{ 0 \}$ which lies in $\sigma_P$ and $\mu>1$ such that $\mu r \in \tau_P$.
  We also know that both $\bar\biglink(\sigma_P)$ and $\bar\biglink(\tau_P)$ are $e$-dimensional.

  Suppose that there exists a point $x\in\bar\biglink(\sigma_P)\setminus\bar\biglink(\tau_P)$.
  The cone $\bar\biglink(\tau_P)$ is full-dimensional in $\{ 0 \}\times \RR^{e}$ and strictly star-shaped \wrt~$0$ by~\autoref{prop:strictly_starshaped}.
  So we obtain $\lambda$ in the open interval $(0,1)$ such that $\lambda x\in\biglink(\tau_P)$.
  Thus, we have found a cell $\tau_Q'\in\biglink(\tau_P)$ with $\lambda x\in\tau_Q'$, and a cell $\sigma_Q'\in\biglink(\sigma_P)$ with $x\in\sigma_Q'$; see \autoref{fig:two_line_crossing:b}.

  If $0\notin \sigma_P$ then, as in the proof of \autoref{prop:strictly_starshaped}, we use \autoref{lemma:two_line_crossing} to show that the cells $\sigma'=\sigma_P\oplus \sigma_Q'$ and $\tau'=\tau_P\oplus\tau_Q'$ have a non-proper intersection;
  they are in fact cells because $\sigma_Q'\in\biglink(\sigma_P)$ and $\tau_Q'\in\biglink(\tau_P)$.
  As we can find an affine hyperplane which separates $\sigma_P$ and $\tau_P$, we know that $r \notin \tau_P$.
  \autoref{lemma:two_line_crossing} applied to $r, \mu r, \lambda x, x$ then yields a non-proper intersection between~$\sigma'$ and~$\tau'$.

  It remains to consider the case $0\in\sigma_P$.
  Then the line segment from the origin to $x$ lies entirely in~$\sigma'$, but only a part of it lies in~$\tau'$, because neither~$0$ nor~$x$ are contained in~$\tau'$.
  We conclude that $\sigma'$ and $\tau'$ do not intersect in a common face, and this final contradiction concludes our proof.
\end{proof}

\subsection{Complementarity of the link map}

To establish this, we need to compare the cone $\bar\biglink$ with the star of the origin in $\triangle_{P\oplus Q}$.
Note that by \equationref{def:sigma_P}, the simplices $\sigma_P$~and~$\sigma_Q$ need not be disjoint, but if their intersection $\sigma_P\cap\sigma_Q$ is non-empty it consists of just the origin.

\begin{lemma}\label{lemma:0_vertex_of_star}
  Let $\sigma=\sigma_P\oplus \sigma_Q$ be a full-dimensional cell in $\st_{\triangle_{P\oplus Q}}(0)$.
    \begin{enumerate}
    \item \label{lemma:0_vertex_of_star:a} If $\sigma_P\cap\sigma_Q\neq\emptyset$ then that intersection contains the
      origin only, and $0$ is a vertex of $\sigma$, $\sigma_P$ and $\sigma_Q$.
    \item \label{lemma:0_vertex_of_star:b} If $\sigma_P\cap\sigma_Q = \emptyset$ then $0$ is not a vertex of $\sigma$,
      and either $0\in \sigma_P$ or $0\in \sigma_Q$.
    \end{enumerate}
\end{lemma}
\begin{proof}
  Suppose that $\sigma_P$ and $\sigma_Q$ intersect non-trivially.
  Then the intersection can only contain the origin as that is the only point which the linear subspaces $\RR^d$ and $\RR^e$ have in common.
  Since $\sigma_P$ and $\sigma_Q$ both are faces of the triangulation $\triangle_{P\oplus Q}$ they need to intersect properly.
  It follows that $0$ is a vertex of both $\sigma_P$ and $\sigma_Q$.
  Hence it is also a vertex of $\sigma$.

  Now let $\sigma_P\cap\sigma_Q=\emptyset$.
  Then $\sigma_P$ and $\sigma_Q$ span mutually skew affine subspaces of $\RR^{d+e}$, and $\sigma$ is an affinely isomorphic image of the affine join of~$\sigma_P$ and~$\sigma_Q$.
  Yet $\sigma_P$ and $\sigma_Q$ are also contained in linear subspaces, $\RR^d$ and $\RR^e$, which are complementary.
  This implies that $\sigma_P$ or $\sigma_Q$ must contain the origin.
  They cannot both contain $0$ since their intersection is empty.
  If $0$ were a vertex of~$\sigma$ it would need to be a vertex of both $\sigma_P$ and $\sigma_Q$.
\end{proof}

In the case \eqref{lemma:0_vertex_of_star:b} of \autoref{lemma:0_vertex_of_star} we have $0 \in \partial \sigma$, as $\sigma_P$ and $\sigma_Q$ are faces of~$\sigma$.

\begin{proposition}\label{lemma:star:every_full_dim_cell_in_P}
  If $0\in \sigma_P$ holds for one full-dimensional cell~$\sigma$ in the star of the origin, then $0\in\tau_P$  for \emph{every} full-dimensional cell $\tau\in\st_{\triangle_{P\oplus Q}}(0)$.
  Moreover, $\dim\tau_P=d$.  The analogous statements hold if $0\in \sigma_Q$.
\end{proposition}
\begin{proof}
  If the origin is a vertex in $\triangle_{P\oplus Q}$, the statement is trivial because $\dim\sigma_P=d$ and $\dim\sigma_Q=e$ for every $\sigma\in\st_{\triangle_{P\oplus Q}}(0)$ by~\autoref{obs:full-dim}.

  Suppose there exist full-dimensional cells $\sigma$ and $\tau$ in $\st_{\triangle_{P\oplus Q}}(0)$
  with $0\in \sigma_P,\tau_Q$ and $0\not\in \sigma_Q,\tau_P$.
  Because $0$ lies on the boundary of both $\sigma$~and~$\tau$, there exist minimal faces $\sigma_0 \subseteq \sigma_P$ and $\tau_0\subseteq\tau_Q$ which contain the origin.
  Because of \autoref{lemma:0_vertex_of_star} and the fact that $0\not\in\Vert\sigma$ and $0\not\in\Vert\tau$, we know that $\sigma_0 \ne \{ 0 \}$ and $\tau_0 \ne \{ 0 \}$.
  We have $\sigma_0 \cap \tau_0 = \{ 0 \}$ as $\sigma_0$ and $\tau_0$ lie in orthogonal linear subspaces, and therefore
  $\sigma$ and $\tau$ do not intersect in a common face.
  This is absurd since we started with a triangulation~$\triangle_{P\oplus Q}$.

  Suppose that the  second assertion does not hold.  Then $0\in \sigma_P$ but $\dim\sigma_P=d-1$.
  This implies that $\dim\sigma_Q=e$, and $\biglink(\sigma_Q)\subset\triangle_{P\oplus Q}$ is a sphere which contains the origin.
  This is a contradiction to \autoref{prop:strictly_starshaped}.
\end{proof}

\begin{corollary}\label{cor:link_of_star}
  All full-dimensional simplices $\sigma=\sigma_P\oplus\sigma_Q\in\st_{\triangle_{P\oplus Q}}(0)$ satisfy $0\in \sigma_P$, or they all satisfy $0 \in \sigma_Q$.
  Both conditions are satisfied simultaneously if and only if the origin occurs as a vertex of~$\sigma$.
\end{corollary}

The previous result says that the origin always lies in the ``$\sigma_P$-part'' or always in the ``$\sigma_Q$-part'', independent of the choice of the cell $\sigma$.
We say that the origin \emph{lies in the $P$-part} or \emph{in the $Q$-part} of the triangulation, respectively.

\begin{proposition}\label{prop:link_of_star}
  Let $\sigma=\sigma_P\oplus\sigma_Q$ and $\tau = \tau_P\oplus\tau_Q$ both be full-dimensional cells in $\st_{\triangle_{P\oplus Q}}(0)$.
  If $0\in\sigma_P$, then $0\in\tau_P$ and $\biglink(\sigma_P) = \biglink(\tau_P)$. The analogous statements hold if $0\in \sigma_Q$.
\end{proposition}

\begin{proof}
  Assuming $0\in\sigma_P$, we can infer $0\in\tau_P$ from \autoref{lemma:star:every_full_dim_cell_in_P} and \autoref{cor:link_of_star}.
  So it remains to show that $\biglink(\sigma_P) = \biglink(\tau_P)$.

  Remember that $\triangle_{P\oplus Q}|_{\bar\biglink(\sigma_P)} \in \cB_{\triangle_{P\oplus Q}}^{e}(0)$ according to \autoref{prop:strictly_starshaped}.
  Suppose that $\biglink(\sigma_P) \ne \biglink(\tau_P)$.
  Then there exists a ray which intersects $\biglink(\sigma_P)$ and $\biglink(\tau_P)$ in different points, say $x$ and $y=\lambda x$ for $\lambda >0$.
  This gives rise to two full-dimensional cells $\sigma',\tau'\in\st_{\triangle_{P\oplus Q}}(0)$, with $x\in\sigma'$ and $y\in\tau'$.
  These cells do not intersect in a common face, as the line segment between~$0$ and~$x$ is contained in~$\sigma'$
and the line segment between~$0$ and~$\lambda x$ is contained in~$\tau'$.
  The minimal faces which contain those line segments cannot intersect properly, since only the shorter one of those line segments is contained in both minimal faces.
  This contradiction refutes the assumption $\biglink(\sigma_P) \ne \biglink(\tau_P)$, and hence proves the claim.
\end{proof}

We now treat the remaining cells to finally show the complementarity property of $\bar\biglink$.

\begin{proposition}
  \label{prop:compatible}
  Let $\sigma = \sigma_P\oplus\sigma_Q$ and $\tau = \tau_P\oplus\tau_Q$ be cells in $\triangle_{P\oplus Q}^{=d+e}$ such that $\dim\sigma_P=d$, $\dim\tau_Q=e$, and at least one of the two simplices $\sigma$ or $\tau$ does not contain the origin. Then
\[
   \tau_Q \subseteq \bar\biglink(\sigma_P)
   \quad\iff\quad
   \sigma_P \not\subseteq \bar\biglink(\tau_Q).
\]
\end{proposition}
\begin{proof}
  For each $\gamma \in \{ \sigma_P, \tau_Q \}$, we set
    \begin{align*}
      \alpha(\gamma) &:= \triangle_{P\oplus Q}\big|_{\bar\biglink(\gamma)},
    \end{align*}
 so that trivially $\tau_Q \subseteq \bar\biglink(\sigma_P)$ if and only if $\tau_Q \in \alpha(\sigma_P)$, and $\sigma_P \subseteq \bar\biglink(\tau_Q)$ if and only if $\sigma_P \in \alpha(\tau_Q)$.

  To reach a contradiction, suppose that $\tau_Q \in \alpha(\sigma_P)$ and $\sigma_P \in \alpha(\tau_Q)$.
  We choose some non-zero point $x\in \relint(\sigma_P)$.
  Then the ray spanned by $x$ intersects the boundary of $\alpha(\tau_Q)$ exactly once, since $\alpha(\tau_Q)$ is a strictly star shaped ball.
  This intersection point can be written as $\lambda x$ for precisely one $\lambda >1$; see \autoref{fig:compatibility:a}.
  Next we choose some maximal cell $\tilde\tau_P\in \partial \alpha(\tau_Q)$ that contains $\lambda x$,
  and some $y\in \relint(\tau_Q)$ with $\mu y \in \tilde \sigma_Q \in \partial\alpha(\sigma_P)$ for $\mu > 1$; see \autoref{fig:compatibility:b}.

  Since $\tilde\sigma_Q$ lies in the boundary of $\alpha(\sigma_P)$, which is the link of $\sigma_P$ in $\Delta_{P\oplus Q}$, we know that $\tilde\sigma_Q$ and thus also $\sigma_P \oplus \tilde \sigma_Q$ is a cell.
  The latter contains $x$ and $\mu y$.
  Since $x$ lies in the relative interior of $\sigma_P$, all proper convex combinations $t x + (1-t)\mu y$ for $t\in(0,1)$ lie in the relative interior of that cell.
  The same argument applies to $y$ and $\lambda x$ in $\tilde \tau_P \oplus \tau_Q$.

  \autoref{lemma:two_line_crossing} yields a point in the interior of those two cells, which violates the intersection property of the triangulation $\triangle_{P\oplus Q}$; compare \autoref{fig:two_line_crossing:b}.
  A similar argument, with $\sigma_P$ outside of $\bar\biglink(\tau_Q)$ and $\tau_Q$ outside of $\bar\biglink(\sigma_P)$, works for the remaining case $\sigma_P\not\in \alpha(\tau_Q)$ and $\tau_Q\not\in\alpha(\sigma_P)$.
\end{proof}

\begin{figure}
    \centering
    \subcaptionbox{\label{fig:compatibility:a}
      Summand $P$
    }{
      \begin{tikzpicture}[scale=1]
        \draw[opacity=0] (0,-2) -- (0,2);

        \coordinate (o) at (0,0);
        \coordinate (lx) at (0.99,0.99);
        \coordinate (x) at (0.3,0.3);

        \draw[black, thick] (30:2) -- (60:2) -- (120:2) -- (210:2) -- (300:2) -- (350:2) -- (30:2) -- cycle;

        \draw[red!70!black, fill=red!20] ($(o)+(0.2,0.2)$) ellipse (1.3 and 1);

        \fill (o) circle (2pt);

        \draw[blue!70!black, fill=blue!20] ($(o)+(1.1,0.3)$) -- ($(o)+(0,0.5)$) -- ($(o)+(0.5,-0.3)$) -- cycle;

        \draw[blue!70!black, thick, rotate=-30, fill=blue!20] ($(lx)-(.3,-.02)$) rectangle ($(lx)+(.3,-.02)$);

        \draw[black, ->, thick] (o) -- (2,2);
        \fill (lx) circle (2pt);
        \fill (x) circle (2pt);

        \node at ($(o)+(0,-0.3)$) {\footnotesize $0$};
        \node at ($(x)+(0,0.3)$) {\footnotesize $x$};
        \node at ($(lx)+(0,0.3)$) {\footnotesize $\lambda x$};
        \node at ($(x)+(0.3,-0.2)$) {\footnotesize $\sigma_P$};
        \node at ($(o)+(-0.33,-1)$) {\footnotesize $\bar\biglink(\tau_Q)$};
        \node at ($(lx)+(.5,-.1)$) {\footnotesize $\tilde\tau_P$};
      \end{tikzpicture}
    }
    \hspace{.1\linewidth}
    \subcaptionbox{\label{fig:compatibility:b}
      Summand $Q$
    }{
      \begin{tikzpicture}[scale=1]
        \draw[opacity=0] (0,-2) -- (0,2);

        \coordinate (o) at (0,0);
        \coordinate (my) at (-0.99,-0.99);
        \coordinate (y) at (-0.3,-0.3);

        \draw[black, thick] (10:2) -- (50:2) -- (120:2) -- (150:2) -- (190:2) -- (220:2) -- (250:2) -- (333:2) -- (10:2) -- cycle;

        \draw[blue!70!black, fill=blue!20] ($(o)+(-0.2,-0.2)$) ellipse (1.3 and 1);

        \fill (o) circle (2pt);

        \draw[red!70!black, fill=red!20] ($(o)+(-1.1,-0.4)$) -- ($(o)+(0,-0.5)$) -- ($(o)+(-0.6,0.3)$) -- cycle;

        \draw[red!70!black, thick, rotate=-30, fill=blue!20] ($(my)-(.3,-.02)$) rectangle ($(my)+(.3,-.02)$);

        \draw[black, ->, thick] (o) -- (-2,-2);
        \fill (my) circle (2pt);
        \fill (y) circle (2pt);

        \node at ($(o)+(0,0.3)$) {\footnotesize $0$};
        \node at ($(y)+(0,-0.3)$) {\footnotesize $y$};
        \node at ($(my)+(0,-0.3)$) {\footnotesize $\mu y$};
        \node at ($(y)+(-0.3,0.2)$) {\footnotesize $\tau_Q$};
        \node at ($(o)+(0.33,1)$) {\footnotesize $\bar\biglink(\sigma_P)$};
        \node at ($(my)+(-.5,.1)$) {\footnotesize $\tilde\sigma_Q$};
      \end{tikzpicture}
    }
    \caption{\label{fig:compatibility}
      Illustrating the geometric situation of \autoref{prop:compatible}.
    }
\end{figure}
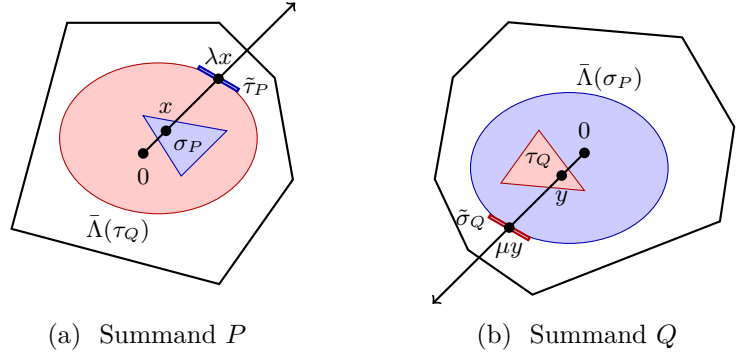

\section{Webs of stars and sum-triangulations}\label{sec:web_of_stars}
\subsection{Webs of stars}
\noindent
This key concept is inspired by the properties of $\bar\biglink$ established in the previous section, and
tells us how to combine the cells of triangulations of the summands to construct triangulations of the free sum.

\begin{definition}[web of stars]
  For any pair of triangulations $\triangle_P$ and $\triangle_Q$ of point configurations $P\subseteq \RR^d$ and $Q\subseteq \RR^e$, a \emph{web of stars} in $\triangle_Q$ with respect to~$\triangle_P$ is an order preserving map
  \[
  \alpha : \big(\triangle_P^{=d}, \preceq\big)
  \to
  \big(\cB_{\triangle_Q}, \subseteq\big).
  \]
  Two webs of stars $\alpha : \triangle_P^{=d}\to \cB_{\triangle_Q}$ and $\beta : \triangle_Q^{=e}\to \cB_{\triangle_P}$ are \emph{compatible} if
  \begin{equation}\label{eq:compatible}
    \sigma \in \beta(\tau) \iff \tau \not\in \alpha(\sigma)
    \qquad\text{for every }
    \sigma\in \triangle_P^{=d}
    \text{ and }
    \tau\in\triangle_Q^{=e}.
  \end{equation}
\end{definition}

\begin{observation}
  \label{obs:compatible}
  By \eqref{eq:compatible}, for each web of stars $\alpha$ there exists at most one compatible web of stars $\beta$ in the reverse direction.
  Essentially, the process of going from~$\alpha$ to~$\beta$ amounts to transposing and complementing the incidence matrix corresponding to~$\alpha$.
\end{observation}

\begin{example}
\label{example:web_of_stars}

  Consider the following two point configurations in $\RR^2$:
    \begin{align*}
      P &= \{ (1,0), (0,1), (-1,1), (-1,0), (1,-1), (0,0) \}, \\
      Q &= \{ (-1,-1), (0,-1), (1,-1), (-1,0), (0,0), (1,0), (-1,1), (0,1), (1,1) \}.
    \end{align*}

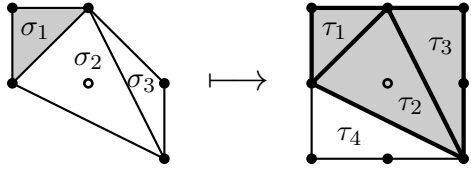
\begin{figure}[bth]
  \centering
  \begin{tikzpicture}[scale=1]
    \tikzstyle{edge} = [draw,thick,-,black]

    \coordinate (e1) at (1,0);
    \coordinate (e2) at (0,1);
    \coordinate (-e1+e2) at (-1,1);
    \coordinate (-e2+e1) at (1,-1);
    \coordinate (-e1) at (-1,0);
    \coordinate (z)  at (0,0);

    \draw[fill=gray!40, draw = none] (-e1) -- (e2) -- (-e1+e2);

    \draw[edge] (e1) -- (e2) -- (-e1+e2) -- (-e1) -- (-e2+e1) -- (e1) -- cycle;
    \draw[edge] (-e2+e1) -- (e2) -- (-e1);

    \foreach \point in {e1,e2,-e1+e2,-e2+e1,-e1,z}
    \fill[black] (\point) circle (2pt);
    
    \fill[white] (z) circle (1pt);

    \node at (-0.7,0.7) {$\sigma_1$};
    \node at (-0,0.3) {$\sigma_2$};
    \node at (0.7,0) {$\sigma_3$};
  \end{tikzpicture}
  \begin{tikzpicture}[scale=1]
    \foreach \x in {-0.7,0.7}
    \foreach \y in {-1,1}
    \fill[opacity=0] (\x,\y) circle (2pt);

    \node at (0,0) {\Large$\longmapsto$};
  \end{tikzpicture}
  \begin{tikzpicture}[scale=1]
    \tikzstyle{edge} = [draw,thick,-,black]

    \coordinate (v0) at (-1,-1);
    \coordinate (v1) at (1,-1);
    \coordinate (v2) at (1,1);
    \coordinate (v3) at (-1,1);
    \coordinate (w1) at (-1,0);
    \coordinate (w2) at (0,1);
    \coordinate (w3) at (0,-1);
    \coordinate (w4) at (1,0);
    \coordinate (z)  at (0,0);

    \draw[fill=gray!40, -, draw = black, ultra thick] (v1) -- (v2) -- (v3) -- (w1) -- (v1) -- cycle;
    \draw[edge, ultra thick] (v1) -- (w2) -- (w1);
    \draw[edge] (v0) -- (v1) -- (v2) -- (v3) -- (v0) -- cycle;

    \foreach \point in {v0,v1,v2,v3,w1,w2,w3,w4,z}
    \fill[black] (\point) circle (2pt);

    \fill[gray!40] (z) circle (1pt);

    \node at (-0.7,0.7) {$\tau_1$};
    \node at (.3,-.3) {$\tau_2$};
    \node at (0.7,0.5) {$\tau_3$};
    \node at (-0.5,-0.7) {$\tau_4$};
  \end{tikzpicture}
  \caption{\label{fig:example:web_of_stars}
    Two triangulations $\triangle_P$ and $\triangle_Q$ and the strictly star shaped ball associated to $\sigma_1$
  }
\end{figure}

  \autoref{fig:example:web_of_stars} shows possible triangulations of $P$~and~$Q$, as well as a simplex $\sigma_1 \in \triangle_P$ and its associated strictly star shaped ball.
  Two compatible webs of stars for these triangulations are
  \begin{align*}
    \alpha : \triangle_P^{=2} &\to \cB_{\triangle_Q}  &  \beta : \triangle_Q^{=2} &\to \cB_{\triangle_P} \\
    \sigma_1 &\mapsto \langle \tau_1,\tau_2,\tau_3 \rangle_{\triangle_Q} &
       \tau_1 &\mapsto \langle \sigma_2,\sigma_3 \rangle_{\triangle_P} \\
    \sigma_2 &\mapsto \langle \tau_2 \rangle_{\triangle_Q} &
       \tau_2 &\mapsto \emptyset \\
    \sigma_3 &\mapsto \langle \tau_2,\tau_3,\tau_4 \rangle_{\triangle_Q} &
       \tau_3 &\mapsto \langle \sigma_2 \rangle_{\triangle_P} \\
    && \tau_4 &\mapsto \langle \sigma_1,\sigma_2 \rangle_{\triangle_P}
  \end{align*}
  Here, e.g., $\sigma_1 \mapsto \langle \tau_1,\tau_2,\tau_3 \rangle_{\triangle_Q}$ indicates that $\alpha$ maps $\sigma_1$ to the subcomplex of $\triangle_Q$ induced by $\tau_1$, $\tau_2$ and $\tau_3$.
  We observe that $\alpha$ and $\beta$ satisfy the following:
  \begin{itemize}
    \item They send faces to strictly star shaped balls \wrt the origin.
      Hence the name ``web of stars''.

    \item They are order preserving.
        For instance, $\sigma_2 \preceq\sigma_1$ and $\alpha(\sigma_2)\subseteq \alpha(\sigma_1)$.
        On the other hand, $\sigma_1$ and $\sigma_3$ are not comparable, and neither are their images.

    \item As for the compatibility condition, take the face $\tau_2$ as an example.
      It is contained in the image of each face under~$\alpha$, and this agrees with $\beta(\tau_2)=\emptyset$.
  \end{itemize}
\end{example}

\subsection{Sum triangulations}
Now we can describe our main construction.

\begin{definition}[sum-triangulation]\label{def:sum_triang}
  A triangulation $\triangle_{P\oplus Q}$ of the free sum $P\oplus Q$ is called a \emph{$P$-sum-triangulation} of $\triangle_P$ and $\triangle_Q$ if there exists a compatible pair of webs of stars $\alpha : \triangle_P^{=d}\to \cB_{\triangle_Q}$ and $\beta : \triangle_Q^{=e}\to \cB_{\triangle_P}$ with the following properties:
    \begin{enumerate}
      \item\label{def:sum_triang:star}
        The $d$-simplices in the star of $0$ of $P$ are sent to the \emph{entire} star of $0$ in~$Q$:
        \[
           \alpha(\sigma_P) = \st_{\triangle_Q}(0)
           \qquad
           \text{for every }
           \sigma_P \in \st_{\triangle_P}(0)^{=d}.
        \]
        This special role of $P$ motivates the name ``$P$-sum-triangulation''.
      \item \label{def:sum_triang:construction}
        The set of all full-dimensional simplices in $\triangle_{P\oplus Q}$ is obtained by summing each simplex $\sigma_P\subset P$ with all simplices in the boundary of its associated star-shaped ball $\alpha(\sigma_P)$, and (almost) vice versa. More precisely,
        \begin{align*}
          \triangle_{P\oplus Q}^{=d+e}
          =&\phantom{{}\cup{}}\!
          \bigcup_{\sigma_P \in \triangle_P^{=d}} \{ \sigma_P \oplus \sigma_Q \;\mid\; \sigma_Q \in \partial\alpha(\sigma_P)^{=e-1} \}\\
          &\cup
          \bigcup_{\substack{\tau_Q \in \triangle_Q^{=e}\\\beta(\tau_Q)\ne \emptyset}} \{ \tau_P \oplus \tau_Q \;\mid\; \tau_P \in \partial\beta(\tau_Q)^{=d-1} \}.
        \end{align*}
        This union is not as asymmetric as it seems: Condition~\eqref{def:sum_triang:star} and the fact that $\alpha$ is an order preserving map imply that $\alpha(\sigma_P)\ne\emptyset$ always holds.
      \end{enumerate}

  If the roles of $P$ and $Q$ are switched we call $\triangle_{P\oplus Q}$ a \emph{$Q$-sum-triangulation}.
  The triangulation $\triangle_{P\oplus Q}$ is a \emph{sum-triangulation} of $\triangle_P$ and $\triangle_Q$ if it is a $P$-sum-triangulation or a $Q$-sum-triangulation.
\end{definition}

\begin{remark}\label{remark:tauQ_maps_to_empty}
  If $\Delta_{P\oplus Q}$ is a $P$-sum triangulation, then the compatibility and the preservation of the order imply that $\beta(\tau_Q) = \emptyset$ for each $\tau_Q\in \st_{\triangle_Q}(0)^{=e}$.
\end{remark}

\begin{example}
  Consider \autoref{example:web_of_stars} again.
  The webs of stars $\alpha$ and $\beta$ satisfy condition \eqref{def:sum_triang:star} of \autoref{def:sum_triang}, so they yield a $P$-sum-triangulation $\triangle_{P\oplus Q}$ via condition~\eqref{def:sum_triang:construction}.
  Some $4$-dimensional simplices in~$\triangle_{P\oplus Q}$ are the convex hull of $\sigma_2$ and every facet of $\tau_2$, where $\sigma_2$ and $\tau_2$ are embedded in the appropriate linear subspaces of $\RR^4$.
  The triangulation which arises in this way has $24$~facets and $11$~vertices.
\end{example}

\begin{example}
%
\begin{figure}[phtb]
  \centering
  \small
  \subcaptionbox{\label{fig:example:2nd_sum:a}}{
    \begin{minipage}{.45\linewidth}
      \centering
      \begin{tikzpicture}[scale=1]
        \tikzstyle{edge} = [draw,thick,-,black]

        \coordinate (v0) at (-1,0);
        \coordinate (v1) at (0,0);
        \coordinate (v2) at (1,0);
        \coordinate (v3) at (2,0);

        \coordinate (w0) at (0,-1);
        \coordinate (w1) at (0,0);
        \coordinate (w2) at (0,1);

        \draw[edge] (v0) -- (w0) -- (v3) -- (w2) -- cycle;
        \draw[edge] (w0) -- (v2) -- (w2);
        \draw[edge] (w0) -- (w2);
        \draw[edge] (v0) -- (v3);

        \foreach \point in {v0,v1,v2,v3,w0,w2}
        \fill[black] (\point) circle (2pt);
        
        \fill[white] (v1) circle (1pt);
      \end{tikzpicture}
      \\
      \begin{align*}
        \triangle_P &= \langle [-1,0],[0,1],[1,2] \rangle\\
        \triangle_Q &= \langle [-1,0],[0,1] \rangle
      \end{align*}
      \begin{align*}
        \alpha([-1,0]) &= \triangle_Q
            & \beta([-1,0]) &= \emptyset\\
        \alpha([0,1]) &= \triangle_Q
            & \beta([0,1]) &= \emptyset\\
        \alpha([1,2]) &= \triangle_Q
      \end{align*}
    \end{minipage}
  }
  \hfill
  \subcaptionbox{\label{fig:example:2nd_sum:b}}{
    \begin{minipage}{.45\linewidth}
      \centering
      \begin{tikzpicture}[scale=1]
        \tikzstyle{edge} = [draw,thick,-,black]

        \coordinate (v0) at (-1,0);
        \coordinate (v1) at (0,0);
        \coordinate (v2) at (1,0);
        \coordinate (v3) at (2,0);

        \coordinate (w0) at (0,-1);
        \coordinate (w1) at (0,0);
        \coordinate (w2) at (0,1);
        
        \fill[gray] (v2) circle (2pt);

        \draw[edge] (v0) -- (w0) -- (v3) -- (w2) -- cycle;
        \draw[edge] (w0) -- (w2);
        \draw[edge] (v0) -- (v3);

        \foreach \point in {v0,v1,v3,w0,w2}
        \fill[black] (\point) circle (2pt);
        
        \fill[white] (v1) circle (1pt);
      \end{tikzpicture}
      \\
      \begin{align*}
        \triangle_P &= \langle [-1,0],[0,2] \rangle\\
        \triangle_Q &= \langle [-1,0],[0,1] \rangle
      \end{align*}
      \begin{align*}
        \alpha([-1,0]) &= \triangle_Q
            & \beta([-1,0]) &= \emptyset\\
        \alpha([0,2]) &= \triangle_Q
            & \beta([0,1]) &= \emptyset\\
        \strut
      \end{align*}
    \end{minipage}
  }
  \\
  \subcaptionbox{\label{fig:example:2nd_sum:c}}{
    \begin{minipage}{.45\linewidth}
      \centering
      \begin{tikzpicture}[scale=1]
        \tikzstyle{edge} = [draw,thick,-,black]

        \coordinate (v0) at (-1,0);
        \coordinate (v1) at (0,0);
        \coordinate (v2) at (1,0);
        \coordinate (v3) at (2,0);

        \coordinate (w0) at (0,-1);
        \coordinate (w1) at (0,0);
        \coordinate (w2) at (0,1);

        \fill[gray] (v1) circle (2pt);  
        \fill[white] (v1) circle (1pt);

        \draw[edge] (v0) -- (w0) -- (v3) -- (w2) -- cycle;
        \draw[edge] (w0) -- (v2) -- (w2);
        \draw[edge] (v0) -- (v3);

        \foreach \point in {v0,v2,v3,w0,w2}
        \fill[black] (\point) circle (2pt);
      \end{tikzpicture}
      \\
      \begin{align*}
        \triangle_P &= \langle [-1,1],[1,2] \rangle\\
        \triangle_Q &= \langle [-1,0],[0,1] \rangle
      \end{align*}
      \begin{align*}
        \alpha([-1,1]) &= \triangle_Q
            & \beta([-1,0]) &= \emptyset\\
        \alpha([1,2]) &= \triangle_Q
            & \beta([0,1]) &= \emptyset\\
        \strut
      \end{align*}
    \end{minipage}
  }
  \hfill
  \subcaptionbox{\label{fig:example:2nd_sum:d}}{
    \begin{minipage}{.45\linewidth}
      \centering
      \begin{tikzpicture}[scale=1]
        \tikzstyle{edge} = [draw,thick,-,black]

        \coordinate (v0) at (-1,0);
        \coordinate (v1) at (0,0);
        \coordinate (v2) at (1,0);
        \coordinate (v3) at (2,0);

        \coordinate (w0) at (0,-1);
        \coordinate (w1) at (0,0);
        \coordinate (w2) at (0,1);

        \fill[gray] (v1) circle (2pt);  
        \fill[white] (v1) circle (1pt);

        \draw[edge] (v0) -- (w0) -- (v3) -- (w2) -- cycle;
        \draw[edge] (w0) -- (v2) -- (w2);
        \draw[edge] (w0) -- (w2);
        \draw[edge] (v2) -- (v3);

        \foreach \point in {v0,v2,v3,w0,w2}
        \fill[black] (\point) circle (2pt);
      \end{tikzpicture}
      \\
      \begin{align*}
        \triangle_P &= \langle [-1,0],[0,1],[1,2] \rangle\\
        \triangle_Q &= \langle [-1,1] \rangle
      \end{align*}
      \begin{align*}
        \alpha([-1,0]) &= \emptyset
            & \beta([-1,1]) &= \st_{\triangle_P}(0)\\
        \alpha([0,1]) &= \emptyset\\
        \alpha([1,2]) &= \triangle_Q
      \end{align*}
    \end{minipage}
  }
  \\
  \subcaptionbox{\label{fig:example:2nd_sum:e}}{
    \begin{minipage}{.45\linewidth}
      \centering
      \begin{tikzpicture}[scale=1]
        \tikzstyle{edge} = [draw,thick,-,black]

        \coordinate (v0) at (-1,0);
        \coordinate (v1) at (0,0);
        \coordinate (v2) at (1,0);
        \coordinate (v3) at (2,0);

        \coordinate (w0) at (0,-1);
        \coordinate (w1) at (0,0);
        \coordinate (w2) at (0,1);

        \fill[gray] (v2) circle (2pt);
        \fill[gray] (v1) circle (2pt);
        \fill[white] (v1) circle (1pt);

        \draw[edge] (v0) -- (w0) -- (v3) -- (w2) -- cycle;
        \draw[edge] (v0) -- (v3);

        \foreach \point in {v0,v3,w0,w2}
        \fill[black] (\point) circle (2pt);
      \end{tikzpicture}
      \\
      \begin{align*}
        \triangle_P &= \langle [-1,2] \rangle\\
        \triangle_Q &= \langle [-1,0],[0,1] \rangle
      \end{align*}
      \begin{align*}
        \alpha([-1,2]) &= \triangle_Q
            & \beta([-1,0]) &= \emptyset\\
            && \beta([0,1]) &= \emptyset
      \end{align*}
    \end{minipage}
  }
  \hfill
  \subcaptionbox{\label{fig:example:2nd_sum:f}}{
    \begin{minipage}{.45\linewidth}
      \centering
      \begin{tikzpicture}[scale=1]
        \tikzstyle{edge} = [draw,thick,-,black]

        \coordinate (v0) at (-1,0);
        \coordinate (v1) at (0,0);
        \coordinate (v2) at (1,0);
        \coordinate (v3) at (2,0);

        \coordinate (w0) at (0,-1);
        \coordinate (w1) at (0,0);
        \coordinate (w2) at (0,1);

        \fill[gray] (v2) circle (2pt);
        \fill[gray] (v1) circle (2pt);
        \fill[white] (v1) circle (1pt);

        \draw[edge] (v0) -- (w0) -- (v3) -- (w2) -- cycle;
        \draw[edge] (w0) -- (w2);

        \foreach \point in {v0,v3,w0,w2}
        \fill[black] (\point) circle (2pt);
      \end{tikzpicture}
      \\
      \begin{align*}
        \triangle_P &= \langle [-1,0],[0,2] \rangle\\
        \triangle_Q &= \langle [-1,1] \rangle
      \end{align*}
      \begin{align*}
        \alpha([-1,0]) &= \emptyset
            & \beta([-1,1]) &= \triangle_P\\
        \alpha([0,2]) &= \emptyset
      \end{align*}
    \end{minipage}
  }
  \caption{\label{fig:example:2nd_sum}
    All triangulations of the sum $P\oplus Q$ where the summands are $P =\{ -1, 0, 1, 2 \}$ and $Q =\{-1, 0, 1\}$, with a representation as a sum-triangulation.
\subref{fig:example:2nd_sum:a}, \subref{fig:example:2nd_sum:b}, \subref{fig:example:2nd_sum:c} and \subref{fig:example:2nd_sum:e} are $P$-sum-triangulations, whereas \subref{fig:example:2nd_sum:d} and \subref{fig:example:2nd_sum:f} are $Q$-sum-triangulations.
  }
\end{figure}
%
  Consider the point sets $P = \{ -1, 0, 1, 2 \}$ and $Q = \{ -1, 0, 1 \}$ in the real line $\RR^1$.
  Every triangulation of $P\oplus Q$ is a sum-triangulation. \autoref{fig:example:2nd_sum} lists them all, together with the corresponding triangulations $\triangle_P$~and~$\triangle_Q$ as well as the compatible webs of stars $\alpha$~and~$\beta$. From the picture we see that \subref{fig:example:2nd_sum:a}, \subref{fig:example:2nd_sum:b}, \subref{fig:example:2nd_sum:c} and \subref{fig:example:2nd_sum:e} are $P$-sum-triangulations, whereas \subref{fig:example:2nd_sum:d} and \subref{fig:example:2nd_sum:f} are $Q$-sum-triangulations.

  A triangulation need not have a unique representation as a $P$- or $Q$-sum triangulation.
Consider $\triangle_P = \langle [-1,0],[0,1],[1,2] \rangle$ and $\triangle_Q = \langle [-1,0],[0,1] \rangle$. Then the triangulation in \subref{fig:example:2nd_sum:a}
of $P\oplus Q$ arises as a P-sum triangulation via the web of stars
      \begin{align*}
        \alpha([-1,0]) &= \triangle_Q
            & \beta([-1,0]) &= \emptyset\\
        \alpha([0,1]) &= \triangle_Q
            & \beta([0,1]) &= \emptyset\\
        \alpha([1,2]) &= \triangle_Q,
\intertext{or as a Q-sum triangulation via the web of stars}
      \tilde\alpha([-1,0]) &= \emptyset
            & \tilde\beta([-1,0]) &= \st_{\triangle_P}(0)\\
      \tilde\alpha([0,1]) &= \emptyset
            & \tilde\beta([0,1]) &= \st_{\triangle_P}(0)\\
      \tilde\alpha([1,2]) &= \triangle_Q.
    \end{align*}
\end{example}

\begin{example}
  The point configuration $P\oplus Q$ from the previous example can be seen as a bipyramid over $P$.
  More generally, suppose that one of the summands --~say $Q$~-- consists of the vertices of a simplex together with the origin as an interior point.
  Then the poset $\cB_{\triangle_Q}(0)$ just consists of the empty set~$\emptyset$ and the complete triangulation~$\triangle_Q$.
  In this case, for fixed $\triangle_P$ and $\triangle_Q$ there exists only one pair of compatible web of stars that produces a $P$-sum-triangulation, and only one pair that produces a $Q$-sum-triangulation.
  For the $P$-sum-triangulation it is
    \begin{align*}
      (\forall \sigma \in \triangle_P) \quad \alpha(\sigma) &= \triangle_Q
            & (\forall \tau \in \triangle_Q) \quad \beta(\tau) &= \emptyset,
  \intertext{while the $Q$-sum-triangulation corresponds to the web of stars}
      (\forall \sigma \in \st_{\triangle_P}(0)) \quad \tilde\alpha(\sigma) &= \emptyset
            & (\forall \tau \in \triangle_Q) \quad \tilde\beta(\tau) &= \st_{\triangle_P}(0)\\
      (\forall \sigma \not\in \st_{\triangle_P}(0)) \quad \tilde\alpha(\sigma) &= \triangle_Q.
    \end{align*}
\end{example}

\subsection{Every triangulation of the free sum is a sum triangulation}
We now come to the most important result in this section.
In Section~\ref{sec:constructing} we showed that any triangulation $\triangle_{P\oplus Q}$ induces (not necessarily unique) triangulations $\triangle_P$ and $\triangle_Q$ of $P$ and $Q$.
Now we will construct the associated webs of stars $\alpha : \triangle_P^{=d}\to \cB_{\triangle_Q}$ and $\beta : \triangle_Q^{=e}\to \cB_{\triangle_P}$ that turn $\triangle_{P\oplus Q}$ into a sum-triangulation.

As the definition of a sum-triangulation \autoref{def:sum_triang} is inspired by the results we showed in \autoref{sec:structure}, we obtain the following result.
Recall that the distinction between the $P$- and the $Q$-part of a free sum triangulation is the topic of \autoref{cor:link_of_star}.
\begin{theorem}\label{thm:sum-triangulation}
  Let $\triangle_{P \oplus Q}$ be any triangulation of $P\oplus Q$ such that the origin lies in the $P$-part.
  Further, let $\triangle_P$ and $\triangle_Q$ be the triangulations obtained by~\autoref{cons:refine-triangulation}.
  Then the maps defined by
  \begin{equation}\label{eq:web_of_stars}
    \begin{aligned}
      \alpha(\sigma_P) \ &:= \ \triangle_{P\oplus Q}\big|_{\bar\biglink(\sigma_P)} \enspace , \\
      \beta(\sigma_Q) \ &:= \
      \begin{cases}
        \emptyset & \text{\upshape if } 0\in\Vert\sigma_Q, \\
        \triangle_{P\oplus Q}\big|_{\bar\biglink(\sigma_Q)} &         \text{\upshape else}
      \end{cases}
    \end{aligned}
  \end{equation}
  form a compatible pair of webs of stars.
  In particular, $\triangle_{P \oplus Q}$ is the $P$-sum-triangulation with respect to $\alpha$ and $\beta$.
\end{theorem}

Clearly, if the origin lies in the $Q$-part, in the above the roles of $P$ and $Q$ are interchanged, and we have a $Q$-triangulation of the free sum.

\begin{proof}
  There are three things to show: The maps $\alpha$, $\beta$ are
\begin{itemize}
  \item well-defined, i.e., their images are strictly star-shaped balls (\autoref{prop:strictly_starshaped});
  \item order preserving (\autoref{prop:order_preserving});
  \item and compatible (the case distinction for $\beta$ together with \autoref{prop:compatible}).
\end{itemize}
\end{proof}

\section{Every sum-triangulation triangulates the free sum}
\label{sec:sum-triangs_are_triangs}
\noindent
From any two triangulations $\triangle_P$ and $\triangle_Q$ of our point configurations $P$ and~$Q$, we wish to construct sum-triangulations of $P\oplus Q$.
In our description we will end up with a $P$-sum-triangulation.
To obtain a $Q$-sum triangulation the roles of $P$ and $Q$ need to be interchanged.

Before we start we need to discuss an issue concerning our notation.
So far we considered the situation where we already have a triangulation of the free sum.
In this case, via the definition \eqref{def:sigma_P}, each simplex in the triangulation of the free sum gives rise to one simplex in each summand.
In Section \ref{sec:constructing} it was shown that this yields triangulations of both summands.
If we want to revert this procedure it is clear, however, that not each simplex in $\triangle_P$ can be matched with any other simplex in $\triangle_Q$ in order to arrive at a simplex of sum triangulation.
For instance, in Figure~\ref{fig:example:2nd_sum:a} the $1$-simplex $[1,2]$ in $\triangle_P$ cannot be paired with the $1$-simplex $[0,1]$ in $\triangle_Q$.
Nonetheless, in order to avoid more cumbersome notation, in this section we write $\sigma_P$, $\sigma_Q$ to denote some simplices in $\triangle_P$ and  $\triangle_Q$, respectively, and we also write $\sigma := \sigma_P\oplus\sigma_Q$ for their direct sum, even though a priori there is no free sum triangulation with a ``common'' pre-image $\sigma$ from which $\sigma_P$,~$\sigma_Q$ could have descended via \eqref{def:sigma_P}.

Throughout this section we pick a web of stars $\alpha:\triangle_P^{=d}\to\cB_{\triangle_Q}$ that satisfies the condition \eqref{def:sum_triang:star} of \autoref{def:sum_triang}.
The map $\beta : \triangle_Q^{=e}\to \cB_{\triangle_P}$ constructed from $\alpha$ via \autoref{obs:compatible} automatically satisfies the compatibility condition with respect to~$\alpha$.
We assume that $\beta$ is itself a web of stars.
To have a concise name for this situation we say that the web of stars $\alpha$ is \emph{proper}.

The following example shows that proper webs of stars always exist.
\begin{example}
  For any triangulations $\triangle_P,\triangle_Q$, we can choose
  \begin{align*}
    \alpha(\sigma_P) \ &:= \ \st_{\triangle_Q}(0) \enspace ,
    \\
    \beta(\sigma_Q) &:= \
    \begin{cases}
      \emptyset & \text{if } \sigma_Q \in \st_{\triangle_Q}(0) \\
      \triangle_P & \text{else}
    \end{cases}
  \end{align*}
  for $\sigma_P\in\triangle_P$ and $\sigma_Q\in\triangle_Q$.
  To see that $\alpha$ and $\beta$ are compatible webs of stars, first note that   $\st_{\triangle_Q}(0)$ is a strictly star shaped ball with respect to~$0$   by~\autoref{lemma:strictly_star_shaped_star}.  The same holds for the entire complex~$\triangle_P$,   because $\bigcup\triangle_P=\conv P$ is convex and therefore strictly star shaped with respect to   the interior point~$0$.  Moreover, $\alpha$~is order-preserving because it is constant, and   $\beta$~is order-preserving because the full dimensional simplices in $\st_{\triangle_Q}(0)$ are   $\preceq$-minimal, and are mapped to the smallest element $\emptyset\in\cB_{\triangle_Q}$.  The compatibility condition is immediate.
\end{example}

Given triangulations $\triangle_P$, $\triangle_Q$ of the summands, from our pair of compatible webs of stars $\alpha:\triangle_P^{=d}\to\cB_{\triangle_Q}$, $\beta:\triangle_Q^{=e}\to\cB_{\triangle_P}$ we triangulate the direct sum $P\oplus Q$ by combining each maximal simplex in one of the summands with all maximal simplices in the boundary of the corresponding star-shaped ball:
  \begin{equation}\label{eq:construction}
    \begin{aligned}
      \cT = \cT(\triangle_P, \triangle_Q, \alpha) \
          :=&\phantom{{}\cup{}}\!
      \bigcup_{\sigma_P \in \triangle_P^{=d}}
      \left\{ \sigma_P \oplus \sigma_Q \;\mid\; \sigma_Q \in \big(\partial\alpha(\sigma_P)\big)^{=e-1} \right\}
         \\
          &\cup
          \bigcup_{\substack{\tau_Q \in \triangle_Q^{=e}\\\beta(\tau_Q)\ne \emptyset}}
      \left\{ \tau_P \oplus \tau_Q \;\mid\; \tau_P \in \big(\partial\beta(\tau_Q)\big)^{=d-1} \right\}\enspace.
    \end{aligned}
  \end{equation}
Recall that condition~\eqref{def:sum_triang:star} of \autoref{def:sum_triang} requires that $\alpha(\sigma_P) = \st_{\triangle_Q}(0)$ for every $\sigma_P \in \st_{\triangle_P}(0)^{=d}$.
By construction each cell in $\cT$ is a $(d+e)$-simplex.
Our goal is to prove that the simplices in~$\cT$ are the maximal cells of a triangulation of $P\oplus Q$.
This requires to show first that any two simplices meet in a common face (which may be empty), and second that they cover the convex hull of the free sum.

\subsection{Any two cells in \texorpdfstring{$\cT$}{T} meet in a common face}
\label{subsec:intersection}
It suffices to show that every pair of distinct maximal cells $\sigma:=\sigma_P \oplus \sigma_Q$ and $\tau:=\tau_P\oplus \tau_Q$ in~$\cT$ can be weakly separated by a hyperplane~$J$ that satisfies
\[
\sigma \cap J \ = \ \sigma \cap \tau \ = \ J \cap \tau \enspace .
\]
It follows that $\sigma$ and $\tau$ do not intersect in their relative interiors, but in a common (maybe trivial) face with supporting hyperplane~$J$.

We will assemble~$J$ from a hyperplane $J_P\subset\RR^d$ that separates~$\sigma_P$ from~$\tau_P$, and a hyperplane $J_Q\subset\RR^e$ that separates~$\sigma_Q$ from~$\tau_Q$.
The bare existence of such separating hyperplanes is obvious since $\triangle_P$~and~$\triangle_Q$ are triangulations, but the crucial step is that the existence of the webs of stars $\alpha,\beta$ ensures that they can be chosen in a compatible way: 

\begin{lemma}\label{lemma:existence_of_separation}
  Let $\sigma_P\oplus\sigma_Q$ and $\tau_P\oplus\tau_Q$ be two cells in $\cT=\cT(\triangle_P,\triangle_Q,\alpha)$.
  If $\sigma_P\ne\tau_P$ and $\sigma_Q\ne\tau_Q$, there exist a hyperplane
  \[
    J_P \ = \ \{ x\in\RR^d \;\mid\; a^T\! x = b \}\subset\RR^d
  \]
  that separates~$\sigma_P$ from~$\tau_P$, and a hyperplane
  \[
    J_Q \ = \ \{ y\in\RR^e \;\mid\; c^T\! y = b \}\subset\RR^e
  \]
  with the same right-hand side that separates~$\sigma_Q$ from~$\tau_Q$, such that
  \[
    J \ = \ \Big\{ (x,y)\in\RR^{d+e} \;\mid\; (a^T, c^T) \begin{pmatrix} x\\y \end{pmatrix} = b \Big\} \subset\RR^{d+e}
  \]
  separates $\sigma$ from $\tau$, i.e., $\sigma_P$~lies in the same half-space of~$J_P$ as the origin in~$\RR^d$ if and only if $\sigma_Q$~lies in the same half-space of~$J_Q$ as the origin in~$\RR^e$.
\end{lemma}

\begin{proof}
We distinguish several cases.

\begin{enumerate}
  \item 
If $\sigma_P$ and $\tau_P$ are not $\preceq$-comparable and $0\in\aff(\sigma_P\cap\tau_P)$, there exists a linear hyperplane $J_P$ separating them.
On the other hand, Lemma~\ref{lem:zero_implies_linear} below yields a linear hyperplane $J_Q$ separating $\sigma_Q$ from $\tau_Q$, and we can form~$J$ by perhaps adjusting the direction of one of the normal vectors~$a,c$.

\item 
If $\sigma_P$ and $\tau_P$ are not $\preceq$-comparable but $0\notin\aff(\sigma_P\cap\tau_P)$, we first fix some hyperplane $J_Q\subset\RR^e$ that separates~$\sigma_Q$ from~$\tau_Q$, whose existence follows from the fact that $\triangle_Q$ is a triangulation.
We now keep track of which half-space of~$J_Q$ contains the origin in~$\RR^e$, or if $0\in J_Q$.
Back in $\RR^p$, by the incomparability of $\sigma_P$ and $\tau_P$ we can find some linear or affine hyperplane~$J_P$ that separates them;
and by perturbing it we can even make $\sigma_P$, $\tau_P$ and $0\in\RR^p$ exhibit the same behavior with respect to~$J_P$ as $\sigma_Q$, $\tau_Q$ and $0\in\RR^e$ have with respect to~$J_Q$.
Further scaling the equation of~$J_P$ to obtain the same right-hand side~$b$ as the equation for~$J_Q$ then yields~$J$.

\item
If $\sigma_P$ and $\tau_P$ are $\preceq$-comparable but $\sigma_Q$ and~$\tau_Q$ are not, we can repeat the arguments of~(1) and~(2) with the roles of $P$~and~$Q$ interchanged.

\item
This leaves us with the case that both $\sigma_P$ and $\tau_P$ as well as $\sigma_Q$ and~$\tau_Q$ are comparable.
From Lemma~\ref{lemma:if_prec_then_not_succ} below it follows that either $\sigma_P\preceq\tau_P$ and $\sigma_Q\preceq\tau_Q$, or $\sigma_P\succeq\tau_P$ and $\sigma_Q\succeq\tau_Q$; we may assume the first case.
But then the definition of~$\preceq$ gives us the required orientations of~$J_P$ and~$J_Q$.
\end{enumerate}
\end{proof}

We now establish the two results we used in the previous proof, before showing the intersection property in Proposition~\ref{prop:intersection}.

\begin{lemma}\label{lem:zero_implies_linear}
  Let $\sigma=\sigma_P\oplus\sigma_Q$ and $\tau=\tau_P\oplus\tau_Q$ be cells of $\cT$.
  If $0\in\aff(\sigma_P\cap\tau_P)$, there exists a linear hyperplane $J_Q$ separating $\sigma_Q$ from $\tau_Q$. 
\end{lemma}
\begin{proof}
There are two cases to consider: The cells $\sigma$ and $\tau$ belong to the same component in~\eqref{eq:construction}, say the first one, or they lie in different components.

  If they lie in the same component, let $\sigma_P$ and $\tau_P$ be $d$-dimensional cells in~$\triangle_P$, and $\sigma_Q\in\partial\alpha(\sigma_P)$, $\tau_Q\in\partial\alpha(\tau_P)$ be $(e-1)$-dimensional cells in~$\triangle_Q$.
  We first show that $\alpha(\sigma_P)=\alpha(\tau_P)$, by supposing to the contrary that there exists a cell $\gamma\in\alpha(\sigma_P)\setminus\alpha(\tau_P)$, say.
  By the compatibility of the webs of stars~$\alpha$ and~$\beta$, the conditions $\gamma\in\alpha(\sigma_P)$, respectively $\gamma\notin\alpha(\tau_P)$, are equivalent to $\sigma_P\notin\beta(\gamma)$, respectively $\tau_P\in\beta(\gamma)$, so we conclude the existence of an $e$-dimensional cell $\gamma\in\triangle_Q$ such that $\beta(\gamma)$ contains exactly one of~$\sigma_P$,~$\tau_P$ (see Figure~\ref{fig:all_lin_same_star}).
\begin{figure}[htb]
  \centering
  \begin{tikzpicture}[scale=1]
    \draw[opacity=0] (0,-2) -- (0,2);

    \coordinate (o) at (0,0);
    \coordinate (lx) at (0.99,0.99);
    \coordinate (s3) at (0.5,0.5);
    \coordinate (s1) at (1.5,0);
    \coordinate (s2) at (0.5,0);
    \coordinate (x) at (1,-0.4);

    \draw[black, thick] (30:2) -- (60:2) -- (120:2) -- (210:2) -- (300:2) -- (350:2) -- (30:2) -- cycle;
    
    \draw[green!70!black, fill=green!20] (s2) -- (s1)
    to[out=45,in=45] (-0.7,0.7)
    to[out=225,in=135] (-0.7,-0.7)
    to[out=315,in=225] (s2) -- cycle;

    \draw[red!70!black, fill=red!20] (s1) -- (s3) -- (s2) -- cycle;
    \draw[blue!70!black, fill=blue!20] (s1) -- (x) -- (s2) -- cycle;
    
    \fill (o) circle (2pt);

    \node at ($(o)+(0,-0.3)$) {\footnotesize $0$};
    \node at (.8,0.2) {\footnotesize $\tau_P$};
    \node at (1,-0.2) {\footnotesize $\sigma_P$};
    \node at (-0.4,0.5) {\footnotesize $\beta(\gamma)$};
    
    \draw[thin, black, dashed] (-2.3,0) -- (2.5,0);
  \end{tikzpicture}
  \hspace{.1\linewidth}
  \begin{tikzpicture}[scale=1]
    \draw[opacity=0] (0,-2) -- (0,2);
    
    \coordinate (o) at (0,0);
    \coordinate (my) at (-0.99,-0.99);
    \coordinate (y) at (-0.3,-0.3);
    \coordinate (t1) at (-1,-0.3);
    \coordinate (t2) at (0,-0.5);
    \coordinate (t3) at (-0.5,0.2);
    
    \draw[black, thick] (10:2) -- (50:2) -- (120:2) -- (150:2) -- (190:2) -- (200:2) -- (270:2) -- (333:2) -- (10:2) -- cycle;
    
    \draw[blue!70!black, fill=blue!20] ($(t2)+(-0.1,-0.1)$) -- ($(t1)+(0,-0.1)$)
    to[out=135,in=180] (0,1.4)
    to[out=0,in=45] (1.3,0)
    to[out=225,in=0] ($(t2)+(-0.1,-0.1)$) -- cycle;          
    \draw[red!70!black, fill=red!20] ($(o)+(0.5,0.2)$) ellipse (0.7 and 0.5);

    \draw[green!70!black, fill=green!20] (t1) -- (t2) -- (t3) -- cycle;
    
    \fill (o) circle (2pt);
    
    \node at ($(o)+(0,0.3)$) {\footnotesize $0$};
    \node at (-0.5,-0.2) {\footnotesize $\gamma$};
    \node at (0,1) {\footnotesize $\alpha(\sigma_P)$};
    \node at (0.65,0.2) {\footnotesize $\alpha(\tau_P)$};
  \end{tikzpicture}
  \caption{\label{fig:all_lin_same_star}
    If the origin lies in the affine span of $\sigma_P\cap\tau_P$, but $\alpha(\sigma_P)\ne\alpha(\tau_P)$, then there exists a cell $\gamma\in\triangle_Q$ such that $\beta(\gamma)$~is not strictly star-shaped.
  }
\end{figure}
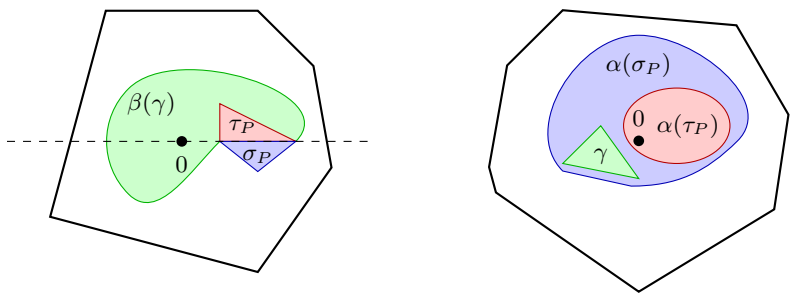
But then $\beta(\gamma)$ cannot be strictly star-shaped, because either $0\in\partial\beta(\gamma)$ (if $\sigma_P\cap\tau_P=\{0\}$), or $\aff(\sigma_P\cap\tau_P)$ contains~$0$ along with two distinct points of $\sigma_P\cap\tau_P$; therefore $\alpha(\sigma_P)=\alpha(\tau_P)$.
Thus, $\sigma_Q$ and $\tau_Q$ are $(e-1)$-dimensional cells on the boundary of the same strictly star-shaped ball, so we conclude the existence of a linear hyperplane separating them.

\medskip
The case where $\sigma$ and $\tau$ belong to different components in~\eqref{eq:construction} cannot in fact arise.
To see this, suppose that $\sigma_P$~is a $d$-dimensional cell in~$\triangle_P$, $\tau_Q$ an $e$-dimensional cell in~$\triangle_Q$, and $\tau_P$ a $(d-1)$-dimensional cell in $\partial\beta(\tau_Q)$.
Now suppose that moreover $0\in\aff(\sigma_P\cap\tau_P)$. Then either $0\in\partial\beta(\tau_Q)$ (if $\sigma_P\cap\tau_P=\{0\})$, or $\aff(\sigma_P\cap\tau_P)$ contains $0$ along with two distinct points of $\sigma_P\cap\tau_P\subset\partial\beta(\tau_Q)$, so that either way $\beta(\tau_Q)$ is not strictly convex. 
\end{proof}

The second result we still need to establish is the following:

\begin{lemma}\label{lemma:if_prec_then_not_succ}
  Let $\sigma_P\oplus\sigma_Q$ and $\tau_P\oplus\tau_Q$ be two cells in $\cT$.
  If $\sigma_P \prec \tau_P$, then either $\sigma_Q\preceq \tau_Q$ holds or $\sigma_Q$ and $\tau_Q$ are not comparable.
\end{lemma}
As usual the roles of $P$ and $Q$ may be interchanged.
\begin{proof}
  Suppose that $\sigma_P\prec \tau_P$ and $\sigma_Q \succ \tau_Q$.
  If $\sigma$~and~$\tau$ are in the first component of~\eqref{eq:construction}, so that $\dim\sigma_P=\dim\tau_P=d$, then $\alpha(\sigma_P) \subseteq \alpha(\tau_P)$ since $\alpha$ is order preserving.
  But because we assume $\sigma_Q \succ \tau_Q$, \autoref{lemma:find_intersecting_ray} gives a ray $r\in\RR^e$ with $\lambda r\in\tau_Q$ and $\mu r \in \sigma_Q$ for some $\mu > \lambda \ge 0$.
  As $\alpha(\sigma_P)$ and $\alpha(\tau_P)$ are strictly star shaped and $\tau_Q \in \partial\alpha(\tau_P)$, we know that $\mu r\not\in\tau_Q$.
  We arrive at $\mu r \not\in 0\star\alpha(\tau_P)$, but $\mu r\in\sigma_Q\in\partial\alpha(\sigma_P)\subset\alpha(\sigma_P)$, which contradicts $\alpha(\sigma_P) \subseteq \alpha(\tau_P)$.
  The same argument works if $\dim\sigma_Q=\dim\tau_Q=e$.

  Now suppose that $\sigma$~and~$\tau$ lie in different components in~\eqref{eq:construction}, so that $\dim\sigma_P=d$ and $\dim\tau_Q=e$.
  As $\alpha(\sigma_P)$ is strictly star shaped and $\sigma_Q$ is a cell in $\alpha(\sigma_P)$, it is clear that $0\star \sigma_Q \subseteq 0 \star \alpha(\sigma_P)$.  From \autoref{lemma:find_intersecting_ray} and our assumption that $\sigma_Q \succ \tau_Q$ we get a stabbing ray that hits $\tau_Q$ first.
  Therefore, $\tau_Q$ has a non-empty intersection with $0\star \sigma_Q$. Because $\alpha(\sigma_P)$ is a strictly star shaped simplicial complex, this implies $\tau_Q\in\alpha(\sigma_P)$.
  With the same argument we get $\sigma_P\in \beta(\tau_Q)$,
  contradicting the compatibility of $\alpha$ and $\beta$.
\end{proof}

We now have all the ingredients together to prove the intersection property of $\cT$.
The idea is to use the two separating hyperplanes from \autoref{lemma:existence_of_separation} and the fact that everything is oriented properly, to construct a ``big'' hyperplane which separates $\sigma, \tau \in \cT$.

\begin{proposition}\label{prop:intersection}
   Any two cells in $\cT$ intersect in a common face.
\end{proposition}
\begin{proof}
  Let $\sigma=\sigma_P\oplus\sigma_Q$ and $\tau=\tau_P\oplus\tau_Q$ be two simplices in $\cT$.

  If $\sigma=\tau$ there is nothing to show.

  If $\sigma_P\ne\tau_P$ and $\sigma_Q\ne\tau_Q$, \autoref{lemma:existence_of_separation} yields
  $H = \{ x\in\RR^d\;\mid\;a^Tx=b \}$ and $G = \{ x\in\RR^e \;\mid\; c^Tx = b \}$ with $0\ne a\in \RR^d$, $0\ne c\in \RR^e$, and $b\in \RR$  such that
    \begin{equation}\label{eq:good_hyperplanes}
      \begin{aligned}
        \sigma_P \subseteq \{ x\in \RR^d \;\mid\; a^Tx &\le b \}\;, &
        \tau_P \subseteq \{ x\in \RR^d \;\mid\; a^Tx &\ge b \}\;, \\
        \sigma_Q \subseteq \{ y\in \RR^e \;\mid\; c^Ty &\le b \}\;, &
        \tau_Q \subseteq \{ y\in \RR^e \;\mid\; c^Ty &\ge b \}\,.
      \end{aligned}
    \end{equation}

  If $\sigma_P = \tau_P$, we choose $a=0$. Since $\sigma_Q \ne \tau_Q$ we can find a separating hyperplane induced by the equation $c^Ty = b$ for some $c\ne 0$.
  Without loss of generality we can assume that $\sigma_Q \subseteq \{ y\in \RR^e \;\mid\; c^Ty \le b \}$ as we can change the sign of $c$ and $b$.

  If $\sigma_Q = \tau_Q$ we choose $c=0$ and find a separating hyperplane induced by the equation $a^Tx = b$ for some $a\ne 0$.
  Again, without loss of generality we can assume that $\sigma_P \subseteq \{ x\in \RR^d \;\mid\; a^Tx \le b \}$ as we can change the sign of $a$ and $b$.

  In all cases, we have found -- not necessarily unique -- $a$, $c$ and $b$ that make \equationref{eq:good_hyperplanes} valid.

  \smallskip
  Next, we perturb and scale $a$ and $c$ such that
    \begin{align}\label{eq:proper_intersection}
      \sigma_P\cap H = &\ \sigma_P\cap \tau_P = H\cap \tau_P\,, &
      \sigma_Q\cap G = &\ \sigma_Q\cap \tau_Q = G\cap \tau_Q\,.
    \end{align}
  Note that $a=0$ if and only if $\sigma_P = \tau_P$, and $c=0$ if and only if $\sigma_Q = \tau_Q$,
  and that our assumption $\sigma\ne\tau$ rules out the case $a=c=0$.
  We may also assume that $b\ge 0$, otherwise we switch the roles of $\sigma$ and $\tau$.
  However, we do not assume that the intersections $\sigma_P\cap\tau_P$ and $\sigma_Q\cap\tau_Q$ are non-empty. If one of the intersections is empty, we perturb, scale, and translate $H$ or $G$ such that \equationref{eq:proper_intersection} is valid.

  Because $(a^T, c^T) \ne 0$, the set
    \[
      J := \{ x\in \RR^{d+e}\;\mid\;(a^T, c^T) x = b\}
    \]
    is a hyperplane in $\RR^{d+e}$. We claim that $J$~separates $\sigma$ from $\tau\in\cT$, and that $\sigma\cap J = \sigma\cap \tau = J\cap \tau$.
  From the definition of $\sigma$ it is clear that every point $s\in\sigma$ can be written as a convex combination of points in $\sigma_P$ and $\sigma_Q$, while
  the definition of $J$ implies that $(a^T, c^T) s \le b$ for all $s\in\sigma$, and that $(a^T, c^T) t \ge b$ for every $t\in\tau$.

  If the two intersections $\sigma_P\cap\tau_P$ and $\sigma_Q\cap\tau_Q$ are empty, then we get $(a^T, c^T) s < b$ for all $s\in\sigma$, and that $(a^T, c^T) t > b$ for every $t\in\tau$. Thus, $J$ strictly separates $\sigma$ and $\tau$.

  Suppose there is a point $x\in (\sigma \cap \tau) \setminus J$.
  Since $x$ is in both $\sigma$ and $\tau$ it can be written both as a convex combination of vertices of $\sigma_P$ and $\sigma_Q$, and as a convex combination of vertices of $\tau_P$ and $\tau_Q$,
    \begin{align*}
      x &= \sum_{v\in \Vert\sigma} \lambda_v v = \sum_{w\in \Vert\tau}\mu_w w.
    \end{align*}
  In order for $x\not\in J$ to hold, there must exist at least one vertex $v$ in $\sigma_P$ or in $\sigma_Q$ with $\lambda_v > 0$, and $(a^T,0)v < b$ if $v\in\sigma_P\times\{0\}$, or $(0,c^T)v < b$ if $v\in\{0\}\times\sigma_Q$.
  In both cases, together with $(a^T,c^T)x\le b$ and $\sum_v\lambda_v=1$ this yields $(a^T, c^T) x < b$.
Hence $x\not\in\tau$ by \equationref{eq:good_hyperplanes}, which contradicts $x\in \sigma\cap \tau$.

  On the other hand, for a point $x\in \sigma \cap J$ that satisfies \equationref{eq:proper_intersection}, a similar argument shows that every vertex $v\in \sigma$ with $\lambda_v>0$ must satisfy $(a^T,0)v = b$ if $v\in\sigma_P\times\{0\}$, or $(0,c^T)v = b$ if $v\in\{0\}\times\sigma_Q$ and therefore has to be in $\tau$.
\end{proof}

\subsection{The cells in \texorpdfstring{$\cT$}{T} cover the free sum}

\begin{proposition}\label{prop:covering}
  We have $\bigcup_{\sigma\in\cT} \sigma = \conv(P\oplus Q)$.
\end{proposition}

\begin{proof}
  We will show that every facet $F$ of any full-dimensional simplex $\sigma = \sigma_P \oplus \sigma_Q$ in $\cT$ is also covered by another simplex in $\cT$, unless $F\subset\partial\cT$.
  Without loss of generality we may assume that $\dim\sigma_P=d$, and that $F = \conv(\Vert\sigma \setminus \{x\})$.
  We distinguish whether the vertex~$x$ we removed is a vertex of $\sigma_P$ or of $\sigma_Q$.

  Suppose $x\in \Vert\sigma_Q$.
  By the definition of $\cT$ we know that $\sigma_Q \in \partial\alpha(\sigma_P)$.
  As the boundary complex of $\alpha(\sigma_P)$ is a triangulated sphere, there exists a neighboring simplex
  \[
  \tilde\sigma_Q =
  \conv\big(
  \Vert{\sigma_Q}\setminus \{ x \} \cup \{  y \}
  \,\big)
  \in
  \partial\alpha(\sigma_P)
  \]
  of $\sigma_Q$ in $\partial\alpha(\sigma_P)$ with respect to $x$.
  Thus, $\sigma_P\oplus\tilde\sigma_Q$ is a simplex in $\cT$, which also covers the facet $F$.

  So let $x\in \Vert\sigma_P$.
  The two cells $f := \conv(\Vert\sigma_P\setminus\{ x \})$ and $\sigma_Q$ are both codimension $1$ faces in the triangulations $\triangle_P$ and $\triangle_Q$ respectively.
  Since $F\not\subset\partial\cT$, at most one of the cells $f,\sigma_Q$ can lie on $\partial\triangle_P$ or $\partial\triangle_Q$.
  We distinguish several cases; see \autoref{fig:covering:setup} for an illustration of the basic setup.

    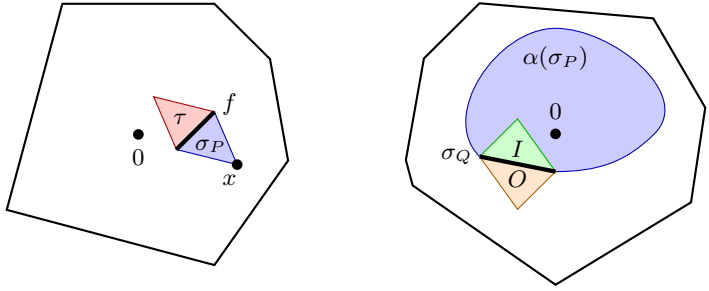
\begin{figure}[htbp]
    \centering
    \begin{tikzpicture}[scale=1]
      \draw[opacity=0] (0,-2) -- (0,2);

      \coordinate (o) at (0,0);
      \coordinate (lx) at (0.99,0.99);
      \coordinate (s3) at (0.2,0.5);
      \coordinate (s1) at (1,0.3);
      \coordinate (s2) at (0.5,-0.2);
      \coordinate (x) at (1.3,-0.4);

      \draw[black, thick] (30:2) -- (60:2) -- (120:2) -- (210:2) -- (300:2) -- (350:2) -- (30:2) -- cycle;
      


      \draw[red!70!black, fill=red!20] (s1) -- (s3) -- (s2) -- cycle;
      \draw[blue!70!black, fill=blue!20] (s1) -- (x) -- (s2) -- cycle;
      
      \draw[black, ultra thick] (s1) -- (s2);
      \node at ($(s1) + (0.2,0.1)$) {\footnotesize $f$};
      
      \fill (o) circle (2pt);
      \fill (x) circle (2pt);

      \node at ($(o)+(0,-0.3)$) {\footnotesize $0$};
      \node at (0.55,0.2) {\footnotesize $\tau$};
      \node at (0.95,-0.15) {\footnotesize $\sigma_P$};
      \node at ($(x) + (-0.1,-0.2)$) {\footnotesize $x$};
    \end{tikzpicture}
    \hspace{.1\linewidth}
    \begin{tikzpicture}[scale=1]
      \draw[opacity=0] (0,-2) -- (0,2);
      
      \coordinate (o) at (0,0);
      \coordinate (my) at (-0.99,-0.99);
      \coordinate (y) at (-0.3,-0.3);
      \coordinate (t1) at (-1,-0.3);
      \coordinate (t2) at (0,-0.5);
      \coordinate (t3) at (-0.5,0.2);
      
      \draw[black, thick] (10:2) -- (50:2) -- (120:2) -- (150:2) -- (190:2) -- (200:2) -- (270:2) -- (333:2) -- (10:2) -- cycle;
      
      \draw[blue!70!black, fill=blue!20] (t2) -- (t1)
      to[out=135,in=180] (0,1.4)
      to[out=0,in=45] (1.3,0)
      to[out=225,in=0] (t2) -- cycle;          

      \draw[green!70!black, fill=green!20] (t1) -- (t2) -- (t3) -- cycle;
      \draw[orange!70!black, fill=orange!20] (t1) -- (t2) -- (-0.5,-1) -- cycle;
      
      \draw[black, ultra thick] (t1) -- (t2);
      \node at ($(t1) + (-0.3,0)$) {\footnotesize $\sigma_Q$};
      \fill (o) circle (2pt);
      
      \node at ($(o)+(0,0.3)$) {\footnotesize $0$};
      \node at (-0.5,-0.2) {\footnotesize $I$};
      \node at (-0.5,-0.6) {\footnotesize $O$};
      \node at (0,1) {\footnotesize $\alpha(\sigma_P)$};
    \end{tikzpicture}
    \caption{\label{fig:covering:setup}
      The basic setup in the proof of \autoref{prop:covering}, where $x\in\sigma_P$. The summand $P$ is on the left and $Q$ is on the right. The cells $\tau$ and $O$ might exist or not, depending on which part of $f\oplus\sigma_Q$ is on the boundary.
    }
  \end{figure}

  \begin{enumerate}
  \item
    Suppose $f\in\partial\triangle_P$, so that $f\in\partial\cT$.
  The codimension~1 cell~$\sigma_Q$ is contained in exactly two adjacent full-dimensional cells in~$\triangle_Q$, and
  one of these, say~$O$, satisfies $O\notin\alpha(\sigma_P)$.
  By compatibility of $\alpha$ and $\beta$, we conclude that $\sigma_P\in\beta(O)$.
  This and the fact that $f\in\partial\cT$ implies that $f\in\partial\beta(O)$.
  Hence $f \oplus O$ is a full-dimensional cell of $\cT$ which contains $F$ as a facet.

  \item
    If $f\notin\partial\triangle_P$,
    we can find a neighboring simplex
    \[
      \tau
      =
      \conv\big(
      \Vert\sigma_P \setminus \{ x \} \cup \{ v \}
      \,\big)
      \in\triangle_P
    \]
    which shares the facet $f$ with $\sigma_P$.
  Depending on whether or not $\sigma_Q\in\partial\triangle_Q$, we know that $\sigma_Q$~is contained in at least one and at most two adjacent full-dimensional cells of $\triangle_Q$.
  One of these (the \emph{inner cell}~$I$) is contained in~$\alpha(\sigma_P)$ and must exist;
  the other one (the \emph{outer cell}~$O$) is not contained in $\alpha(\sigma_P)$ and might exist or not.

\begin{enumerate}
\item  Let us assume that $O$ exists. By compatibility of $\alpha$ and $\beta$, from $I\in\alpha(\sigma_P)$ we conclude $\sigma_P\notin\beta(I)$, and from $O\notin\alpha(\sigma_P)$ that $\sigma_P\in\beta(O)$.
  There are three possible cases:
  
\begin{enumerate}
  \item \label{prop:covering:case:I_O_outside}
     If $I\not \in \alpha(\tau)$,
     compatibility yields $\tau\in\beta(I)$.
     Since moreover $\sigma_P\notin\beta(I)$, the facet $f=\sigma_P\cap\tau$ must be part of~$\partial\beta(I)$.
     Hence, the simplex $f\oplus I\in\cT$ contains~$F$. See \autoref{fig:covering:a}.

     \begin{figure}[htbp]
       \centering
       \begin{tikzpicture}[scale=1]
         \draw[opacity=0] (0,-2) -- (0,2);

         \coordinate (o) at (0,0);
         \coordinate (lx) at (0.99,0.99);
         \coordinate (s3) at (0.2,0.5);
         \coordinate (s1) at (1,0.3);
         \coordinate (s2) at (0.5,-0.2);
         \coordinate (x) at (1.3,-0.4);

         \draw[black, thick] (30:2) -- (60:2) -- (120:2) -- (210:2) -- (300:2) -- (350:2) -- (30:2) -- cycle;
         
         \draw[orange!70!black, fill=orange!20] ($(o)+(0.25,0.2)$) ellipse (1.5 and 1.3);

         \draw[green!70!black, fill=green!20] (s2) -- (s1)
         to[out=45,in=45] (-0.7,0.7)
         to[out=225,in=135] (-0.7,-0.7)
         to[out=315,in=225] (s2) -- cycle;

         \draw[red!70!black, fill=red!20] (s1) -- (s3) -- (s2) -- cycle;
         \draw[blue!70!black, fill=blue!20] (s1) -- (x) -- (s2) -- cycle;
         
         \fill (o) circle (2pt);
         \fill (x) circle (2pt);

         \node at ($(o)+(0,-0.3)$) {\footnotesize $0$};
         \node at (0.55,0.2) {\footnotesize $\tau$};
         \node at (0.9,-0.1) {\footnotesize $\sigma_P$};
         \node at ($(x) + (-0.1,-0.2)$) {\footnotesize $x$};
         \node at (-0.4,0.3) {\footnotesize $\beta(I)$};
         \node at (0.3,1.2) {\footnotesize $\beta(O)$};
       \end{tikzpicture}
       \hspace{.1\linewidth}
       \begin{tikzpicture}[scale=1]
         \draw[opacity=0] (0,-2) -- (0,2);
         
         \coordinate (o) at (0,0);
         \coordinate (my) at (-0.99,-0.99);
         \coordinate (y) at (-0.3,-0.3);
         \coordinate (t1) at (-1,-0.3);
         \coordinate (t2) at (0,-0.5);
         \coordinate (t3) at (-0.5,0.2);
         
         \draw[black, thick] (10:2) -- (50:2) -- (120:2) -- (150:2) -- (190:2) -- (200:2) -- (270:2) -- (333:2) -- (10:2) -- cycle;
         
         \draw[blue!70!black, fill=blue!20] (t2) -- (t1)
         to[out=135,in=180] (0,1.4)
         to[out=0,in=45] (1.3,0)
         to[out=225,in=0] (t2) -- cycle;          
         \draw[red!70!black, fill=red!20] ($(o)+(0.4,0.2)$) ellipse (0.6 and 0.5);

         \draw[green!70!black, fill=green!20] (t1) -- (t2) -- (t3) -- cycle;
         \draw[orange!70!black, fill=orange!20] (t1) -- (t2) -- (-0.5,-1) -- cycle;
         
         \draw[black, ultra thick] (t1) -- (t2);
         \node at ($(t1) + (-0.3,0)$) {\footnotesize $\sigma_Q$};
         \fill (o) circle (2pt);
         
         \node at ($(o)+(0,0.3)$) {\footnotesize $0$};
         \node at (-0.5,-0.2) {\footnotesize $I$};
         \node at (-0.5,-0.6) {\footnotesize $O$};
         \node at (0,1) {\footnotesize $\alpha(\sigma_P)$};
         \node at (0.65,0.2) {\footnotesize $\alpha(\tau)$};
       \end{tikzpicture}
       \caption{Case \ref{prop:covering:case:I_O_outside}.} 
       \label{fig:covering:a}
     \end{figure}
     
  \item \label{prop:covering:case:I_O_inside}
    If $I, O \in \alpha(\tau)$,
     compatibility and $O\in\alpha(\tau)$ yield $\tau\notin\beta(O)$.
     Since moreover $\sigma_P\in\beta(O)$, the facet $f=\sigma_P\cap\tau$ must be part of $\partial\beta(O)$.
    Hence, the simplex $f\oplus O\in\cT$ contains~$F$. See \autoref{fig:covering:b}.

     \begin{figure}[htbp]
       \centering
       \begin{tikzpicture}[scale=1]
         \draw[opacity=0] (0,-2) -- (0,2);

         \coordinate (o) at (0,0);
         \coordinate (lx) at (0.99,0.99);
         \coordinate (x) at (0.2,0.5);
         \coordinate (s1) at (1,0.3);
         \coordinate (s2) at (0.5,-0.2);
         \coordinate (s3) at (1.3,-0.4);

         \draw[black, thick] (30:2) -- (60:2) -- (120:2) -- (210:2) -- (300:2) -- (350:2) -- (30:2) -- cycle;

         \draw[orange!70!black, fill=orange!20] (s2) -- (s1)
         to[out=45,in=45] (-0.8,1.3)
         to[out=225,in=135] (-0.8,-1)
         to[out=315,in=225] (s2) -- cycle;
         
         \draw[green!70!black, fill=green!20] ($(o)+(-0.45,-0.2)$) ellipse (0.6 and 0.5);

         \draw[red!70!black, fill=red!20] (s1) -- (s3) -- (s2) -- cycle;
         \draw[blue!70!black, fill=blue!20] (s1) -- (x) -- (s2) -- cycle;

         \fill (o) circle (2pt);
         \fill (x) circle (2pt);

         \node at ($(o)+(0,-0.3)$) {\footnotesize $0$};
         \node at (0.55,0.2) {\footnotesize $\sigma_P$};
         \node at (0.9,-0.1) {\footnotesize $\tau$};
         \node at ($(x) + (0.1,0.2)$) {\footnotesize $x$};
         \node at (-0.6,-0.1) {\footnotesize $\beta(I)$};
         \node at (-0.3,1.1) {\footnotesize $\beta(O)$};
       \end{tikzpicture}
      \hspace{.1\linewidth}
      \begin{tikzpicture}[scale=1]
        \draw[opacity=0] (0,-2) -- (0,2);

        \coordinate (o) at (0,0);
        \coordinate (my) at (-0.99,-0.99);
        \coordinate (y) at (-0.3,-0.3);
        \coordinate (t1) at (-1,-0.3);
        \coordinate (t2) at (0,-0.5);
        \coordinate (t3) at (-0.5,0.2);

        \draw[black, thick] (10:2) -- (50:2) -- (120:2) -- (150:2) -- (190:2) -- (200:2) -- (270:2) -- (333:2) -- (10:2) -- cycle;

        \draw[red!70!black, fill=red!20] ($(o)+(0.2,0)$) ellipse (1.4 and 1.5);

        \draw[blue!70!black, fill=blue!20] (t2) -- (t1)
        to[out=135,in=180] (0,1)
        to[out=0,in=45] (1,0)
        to[out=225,in=0] (t2) -- cycle;

        \draw[green!70!black, fill=green!20] (t1) -- (t2) -- (t3) -- cycle;
        \draw[orange!70!black, fill=orange!20] (t1) -- (t2) -- (-0.5,-1) -- cycle;

        \fill (o) circle (2pt);

        \draw[black, ultra thick] (t1) -- (t2);
        \node at ($(t1) + (-0.3,0)$) {\footnotesize $\sigma_Q$};

        \node at ($(o)+(0,0.3)$) {\footnotesize $0$};
        \node at (-0.5,-0.2) {\footnotesize $I$};
        \node at (-0.5,-0.6) {\footnotesize $O$};
        \node at (0,0.7) {\footnotesize $\alpha(\sigma_P)$};
        \node at (0.3,-1) {\footnotesize $\alpha(\tau)$};
      \end{tikzpicture}
       \caption{Case \ref{prop:covering:case:I_O_inside}}
       \label{fig:covering:b}
     \end{figure}
    
  \item \label{prop:covering:case:I_inside_O_outside}
    If $I \in \alpha(\tau)$ and $O \not\in \alpha(\tau)$,
    the shared facet $\sigma_Q$ of $I$ and $O$ must be part of both $\partial\alpha(\tau)$ and $\partial\alpha(\sigma_P)$.
    Therefore, $\tau \oplus \sigma_Q$ is a simplex in~$\cT$ and contains~$F$.
    See \autoref{fig:covering:c}.

     \begin{figure}[htbp]
       \centering
       \begin{tikzpicture}[scale=1]
         \draw[opacity=0] (0,-2) -- (0,2);

         \coordinate (o) at (0,0);
         \coordinate (lx) at (0.99,0.99);
         \coordinate (s3) at (0.2,0.5);
         \coordinate (s1) at (1,0.3);
         \coordinate (s2) at (0.5,-0.2);
         \coordinate (x) at (1.3,-0.4);

         \draw[black, thick] (30:2) -- (60:2) -- (120:2) -- (210:2) -- (300:2) -- (350:2) -- (30:2) -- cycle;

         \draw[orange!70!black, fill=orange!20] ($(o)+(0.25,0.2)$) ellipse (1.5 and 1.3);

         \draw[green!70!black, fill=green!20] ($(o)+(-0.45,-0.2)$) ellipse (0.6 and 0.5);

         \draw[red!70!black, fill=red!20] (s1) -- (s3) -- (s2) -- cycle;
         \draw[blue!70!black, fill=blue!20] (s1) -- (x) -- (s2) -- cycle;

         \fill (o) circle (2pt);
         \fill (x) circle (2pt);

         \node at ($(o)+(0,-0.3)$) {\footnotesize $0$};
         \node at (0.55,0.2) {\footnotesize $\tau$};
         \node at (0.9,-0.1) {\footnotesize $\sigma_P$};
         \node at ($(x) + (-0.1,-0.2)$) {\footnotesize $x$};
         \node at (-0.6,-0.1) {\footnotesize $\beta(I)$};
         \node at (0.3,1.2) {\footnotesize $\beta(O)$};
       \end{tikzpicture}
      \hspace{.1\linewidth}
      \begin{tikzpicture}[scale=1]
        \draw[opacity=0] (0,-2) -- (0,2);

        \coordinate (o) at (0,0);
        \coordinate (my) at (-0.99,-0.99);
        \coordinate (y) at (-0.3,-0.3);
        \coordinate (t1) at (-1,-0.3);
        \coordinate (t2) at (0,-0.5);
        \coordinate (t3) at (-0.5,0.2);

        \draw[black, thick] (10:2) -- (50:2) -- (120:2) -- (150:2) -- (190:2) -- (200:2) -- (270:2) -- (333:2) -- (10:2) -- cycle;

        \draw[red!70!black, fill=red!20] (t2) -- (t1)
        to[out=135,in=180] (-0.4,1.4)
        to[out=0,in=45] (1.5,-0.5)
        to[out=225,in=0] (t2) -- cycle;

        \draw[blue!70!black, fill=blue!20] (t2) -- (t1)
        to[out=135,in=180] (0,1)
        to[out=0,in=45] (1,0)
        to[out=225,in=0] (t2) -- cycle;

        \draw[green!70!black, fill=green!20] (t1) -- (t2) -- (t3) -- cycle;
        \draw[orange!70!black, fill=orange!20] (t1) -- (t2) -- (-0.5,-1) -- cycle;

        \fill (o) circle (2pt);
        
        \draw[black, ultra thick] (t1) -- (t2);
        \node at ($(t1) + (-0.3,0)$) {\footnotesize $\sigma_Q$};

        \node at ($(o)+(0,0.3)$) {\footnotesize $0$};
        \node at (-0.5,-0.2) {\footnotesize $I$};
        \node at (-0.5,-0.6) {\footnotesize $O$};
        \node at (0,0.7) {\footnotesize $\alpha(\sigma_P)$};
        \node at (1.2,-0.31) {\footnotesize $\alpha(\tau)$};
      \end{tikzpicture}
       \caption{Case \ref{prop:covering:case:I_inside_O_outside}}
       \label{fig:covering:c}
     \end{figure}
    
  \end{enumerate}

\item
  If $O$ does not exist, we proceed as in \ref{prop:covering:case:I_O_outside} and \ref{prop:covering:case:I_inside_O_outside}.
  For the latter case, we use that $\sigma_Q\in\partial\Delta_Q$ and therefore $\partial\alpha(\tau)$ and $\partial\alpha(\sigma_P)$ both contain~$\sigma_Q$.
\end{enumerate}
\end{enumerate}
\end{proof}
We are finally ready to establish our second main result.
\begin{theorem}\label{thm:webs-yield-triangulation}
  Let $\triangle_P$ and $\triangle_Q$ be triangulations of the point configurations $P$ and $Q$, respectively.
  Further, let $\alpha : \triangle_P^{=d}\to \cB_{\triangle_Q}$ be a web of stars which is proper.
  Then the $(d{+}e)$-simplices in $\cT(\triangle_P, \triangle_Q, \alpha)$, as defined in \eqref{eq:construction}, generate a $P$-sum-triangulation of the free sum $P\oplus Q$.
\end{theorem}
\begin{proof}
  \autoref{prop:intersection} shows that the $\cT=\cT(\triangle_P, \triangle_Q, \alpha)$ generates a finite simplicial complex.
  By \autoref{prop:covering} this simplicial complex covers the entire convex hull of the points in $P\oplus Q$.
\end{proof}

\section{Application: Fano polytopes}
\phantomsection\label{sec:fano}
\noindent
The theory of toric varieties is an area within algebraic geometry which is specially amenable to
combinatorial techniques; see \cite{Book_CoxLittleSchenck}.  This is due to the fact that toric varieties
can be described in terms of face fans (or, dually, normal fans) of lattice polytopes, i.e.,
polytopes whose vertices have integral coordinates.  Of particular interest are the (smooth) Fano
varieties.  A full-dimensional lattice polytope $P$ is a \emph{smooth Fano polytope} if it contains
the origin $0$ as an interior point, and the vertex set of each facet forms a lattice basis.  In
particular, smooth Fano polytopes are simplicial.  Note that in the literature smooth Fano polytopes
are sometimes simple; in that case the polytope $P$ needs to be exchanged with its polar $P^*$.  In
the sequel we will identify a smooth Fano polytope $P$ with its \emph{canonical point configuration}
which consists of the vertices of $P$ plus the origin.  This is precisely the set of lattice points
in $P$.

If $P$ and $Q$ are the canonical point configurations of two smooth Fano polytopes, say of
dimensions $d$ and $e$, then the free sum $P\oplus Q$ is the canonical point configuration of a
smooth Fano polytope of dimension $d+e$.  This is an easy consequence of the fact that each facet of
the sum is the affine join of facets, one from each summand.  Our results on sum triangulations
allow to combine the triangulations of the summands $P$ and $Q$ into a description of all
triangulations of the sum $P\oplus Q$.  This is particularly interesting since it is conjectured
that many smooth Fano polytopes admit a non-trivial splitting
\cite[Conjecture~9]{AssarfJoswigPaffenholz:2014}; a partial solution has been obtained in
\cite[Theorem~1]{AssarfNill:1409.7303}.  In Table~\ref{tab:fano-decomp-stats} we list the number of
Fano polytopes (up to unimodular equivalence) of dimension at most six split by their numbers of
free summands.  The percentage of decomposable ones goes down with the dimension, but slowly.  The
diagonal entries in Table~\ref{tab:fano-decomp-stats}, reporting one class each of $d$-dimensional
Fano polytopes which decompose into $d$ summands, correspond to the cross polytopes $\cross(d)$.

\begin{table}[tbh]
  \centering
  \caption{The number of Fano $d$-polytopes (up to unimodular equivalence) split by number of free summands}
  \label{tab:fano-decomp-stats}
  \renewcommand{\arraystretch}{0.9}
  \begin{tabular*}{\linewidth}{@{\extracolsep{\fill}}ccccccccc@{}}
    \toprule
    &&\multicolumn{6}{c}{\#summands} &\\
    $d$ & total & $1$ & $2$ & $3$ & $4$ & $5$ & $6$ & decomposable\\
    \midrule
    $2$ & $5$ & $4$ & $1$ &&&&& $20\%$ \\
    $3$ & $18$ & $13$ & $4$ & $1$ &&&& $28\%$ \\
    $4$ & $124$ & $96$ & $23$ & $4$ & $1$ &&& $23\%$ \\
    $5$ & $866$ & $690$ & $148$ & $23$ & $4$ & $1$ && $20\%$ \\
    $6$ & $7622$ & $6261$ & $1165$ & $168$ & $23$ & $4$ & $1$ & $18\%$  \\
    \bottomrule
  \end{tabular*}
\end{table}

We will devote the rest of this section to studying the following example.
\begin{example}
  The $d$-dimensional \emph{del Pezzo polytope} $\DP(d)$ is defined as the convex hull of the $d$-dimensional cross
  polytope and the two additional points $\pm\1$.  That is, it has exactly $2d+2$ vertices, which read
  \[
  \pm e_1 \,,\ \pm e_2 \,,\ \dots \,,\ \pm e_d \,,\ \pm\1 \enspace .
  \]
  By construction the del Pezzo polytopes are centrally symmetric.  The \emph{pseudo del Pezzo polytope} $\DP^-(d)$ is
  the subpolytope of $\DP(d)$ which arises as the convex hull of all its vertices except for $\1$.  Both, $\DP(d)$ and
  $\DP^-(d)$, are smooth Fano polytopes, provided that $d$ is even.

  In the sequel we will primarily address the canonical point configuration of the free sum $\DP(2)\oplus \DP(4)$.
  This comprises $10+6+1=17$ lattice points in $\RR^6$.
  We will also consider (the canonical point configuration of) $\DP(2)\oplus \DP^-(4)$, which has only one point less.
\end{example}

Before we continue let us recall the state of the art in the enumeration of triangulations of some point set $P$.  The
most successful general method starts out with computing one triangulation, say $\Delta$, of $P$, e.g., a \emph{placing
  triangulation}, obtained from the beneath-and-beyond convex hull algorithm.  The second step, which is much more
demanding, is to apply a purely combinatorial procedure to obtain all triangulations of $P$ which are connected to
$\Delta$ by a sequence of local modifications known as \emph{(bistellar) flips}.  The algorithm is described in
\cite{PfeifleRambau:2003} and implemented in \topcom~\cite{TOPCOM}.  In general, this method will not enumerate all
triangulations of $P$ but rather only the regular ones along with those connected to a regular triangulation by a
sequence of flips.  Except for the naive, i.e., computationally intractable, approach by combinatorial enumeration from
all subsets of maximal simplices there is no general method known which produces the entire set of all triangulations of
any given point configuration, including the non-regular ones.

It is important that \topcom's flip algorithm can take symmetries into account.  The group of invertible linear maps
which fixes a finite point configuration also acts on the set of all its triangulations.  \topcom can be given a set of
generators of a group as input in addition to the point set; it then produces only one triangulation per orbit of the
induced action. Our example is highly symmetric: the group of linear symmetries of $\DP(4)$ has order $240$, while the
group of linear symmetries of $\DP(2)$, which is a dihedral group, has order $12$.  It follows that our point
configuration $\DP(2)\oplus\DP(4)$ of $17$ points in dimension six admits a group of order $12\cdot 240=2880$, and it
turns out that this is also the entire group of linear symmetries.  Yet, with standard hardware of today, it seems to be
next to impossible to determine the set of triangulations of $\DP(2)\oplus\DP(4)$, even just up to symmetry: After nine
days worth of CPU time our \topcom computation stopped since it reached the imposed memory limit of 26 GB, without
arriving at the result.%
\footnote{\emph{Note added in proof.}
  After this paper was submitted, Lars Kastner used the new parallelized implementation \texttt{mptopcom} (\url{https://polymake.org/doku.php/mptopcom}) to find the number of regular triangulations of~$\DP(2) \oplus \DP(4)$ to be 144\,110. The total number of triangulations remains unknown.}%


\begin{table}[tbh]
  \centering
  \caption{Triangulations
    and homomorphisms encoding compatible pairs of webs of stars for $\DP(2) \oplus
    \DP(4)$.  Smaller point configurations are added for comparison}
  \label{tab:topcom}
  \renewcommand{\arraystretch}{0.9}
  \begin{tabular*}{\linewidth}{@{\extracolsep{\fill}}lrcrc@{}}
    \toprule
    & \#triangulations & time & \#homomorphisms & time \\ 
    \midrule
    $\DP(2) \oplus \DP(2)$ & 204 (155 reg.) & 2s & 1\,157 & 8s \\ 
    \addlinespace
    $\DP(2) \oplus \cross(4)$ & 13 (13 reg.) & 20s & 16 & 11s \\ %
    $\DP(2) \oplus \DP^-(4)$ & 250\,594 (12\,846 reg.) & 2.5h & 1\,581\,647 & 27s \\ 
    $\DP(2) \oplus \DP(4)$ & ? && 1\,677\,949\,075 & 10d \\ 
    \bottomrule
  \end{tabular*}
\end{table}

Our techniques for triangulations of free sums are constructive, and we made a proof-of-concept implementation in
\polymake~\cite{DMV:polymake}, which takes the triangulations of both summands as input.  For our example, \topcom
returns seven triangulations of $\DP(2)$ and $128$ triangulations of $\DP(4)$; where these and all subsequent counts are
up to symmetry.  The time for this computation is about one minute, of which almost everything is spent on the
$4$-dimensional point configuration.  Then, for each triangulation $\Delta_P$ of the hexagon $\DP(2)$, our code computes
the stabbing order among the triangles.  This takes almost no time.  Slightly more costly, with about $20$ minutes, is
the computation of the web of star poset of a triangulation $\Delta_Q$ of $\DP(4)$.  The final step is to compute all
admissible homomorphisms from the stabbing poset of $\Delta_P$ into the web of sphere poset of $\Delta_Q$, again up to
symmetry.  This took us $10$ days.  The total number of triangulations of $\DP(2)\oplus\DP(4)$ seems to be huge (we feel
pretty safe in guessing that it exceeds 100 million).  Hence, for lack of memory, we refrained from explicitly
constructing the triangulations.  Since one triangulation can be obtained from more than one homomorphism, we arrive at
an overcount this way.

Interestingly, to compute the same for the subpolytope $\DP(2)\oplus\DP^-(4)$, with only one vertex less, is fairly
easy.  Phrased differently, the smooth Fano polytope $\DP(2)\oplus\DP(4)$ is an example for which standard techniques
fail by a small margin only.  It is this realm where our specialized approach seems to be most useful.

Moreover, we believe that the data shown underestimates the potential of our methods for computing triangulations of
smooth Fano polytopes.  One reason is that very many polytopes which are listed as indecomposable in
Table~\ref{tab:fano-decomp-stats} are (possibly iterated) skew bipyramids over lower-dimensional smooth Fano polytopes;
see \cite[Lemma~3]{AssarfJoswigPaffenholz:2014}.  We expect that webs of stars and stabbing orders can be applied to
their triangulations, too.  However, this is beyond our scope here.  Moreover, as far as timings are concerned, our
proof-of-concept implementation in \polymake leaves room for improvements.  For instance, not even straightforward
parallelization is employed.

\section{For further research}

\addtocontents{toc}{\SkipTocEntry}
\subsection{More than two summands}
Until now we only considered the free sum of \emph{two} summands.
But since $\oplus$ is associative, we can generalize our results to the free sum of finitely many polytopes
\[
P_1 \oplus P_2 \oplus P_3 \oplus \cdots \oplus P_k = (\cdots ((P_1 \oplus P_2) \oplus P_3) \oplus \cdots \oplus P_k) \enspace .
\]
By a repeated application of our characterization the triangulations of the summands $P_1$ up to $P_k$ contain enough information to describe every triangulation of the multiple sum.

\addtocontents{toc}{\SkipTocEntry}
\subsection{Subfree sum}

In \cite{McMullen1976} McMullen introduced the \emph{subfree sum}, which generalizes the free sum by allowing the origin to lie on the boundary of the participating polytopes.
The results in this paper should largely translate to the subfree sum, but there are some pitfalls.
For example, combining a simplex $\sigma_P$ in $\triangle_P$ with the boundary of the web of stars $\alpha(\sigma_P)$ can yield non-proper intersections.
One should only combine $\sigma_P$ with \emph{some} faces on the boundary;
using only those cells which ``face away'' from the origin should yield a correct choice.

But more subtle changes are needed to generalize our results. We suspect that, with some effort, the concepts introduced here can be extended to the case when the origin is not contained in the summands, and to arbitrary subdivisions.

\addtocontents{toc}{\SkipTocEntry} 
\subsection{Regularity}

A triangulation is \emph{regular} (or \emph{coherent}) if it is induced be a height function.
For applications in algebraic geometry such triangulations are the most interesting ones.
Therefore, it is of major interest to characterize those triangulations of a free sum which are regular.

\begin{conjecture}[regularity]\label{conj:regular}
  A sum-triangulation $\triangle_{P\oplus Q}$ of $P\oplus Q$ determined by $\triangle_P,\triangle_Q$ and compatible webs of stars   $\alpha : \triangle_P^{=d}\to \cB_{\triangle_Q}$ and $\beta : \triangle_Q^{=e}\to \cB_{\triangle_P}$ is regular if and only if the triangulations of the summands are regular and the images of~$\alpha$ and~$\beta$ are totally ordered.
  In other words: for every pair of cells $\sigma_P, \tau_P \in \triangle_P$, one of the conditions $\alpha(\sigma_P)\subseteq \alpha(\tau_P)$ or $\alpha(\sigma_P)\supseteq \alpha(\tau_P)$ must hold.
\end{conjecture}

Notice that the conjecture deduces a geometric property from  a purely combinatorial condition.
In fact, what the conjecture implicitly states is that the free sum is geometrically restrictive enough to warrant this, the reason being that the summands are embedded in \emph{mutually orthogonal} linear subspaces.
To give an intuition as for why the images of $\alpha$ and $\beta$ should be totally ordered, consider the following example.

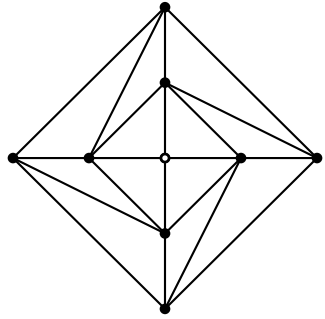
\begin{figure}[htb]
  \centering
  \begin{tikzpicture}[scale=1]
    \tikzstyle{edge} = [draw,thick,-,black]

    \coordinate (v0) at (-2,0);
    \coordinate (v1) at (-1,0);
    \coordinate (v2) at (0,0);
    \coordinate (v3) at (1,0);
    \coordinate (v4) at (2,0);

    \coordinate (w0) at (0,-2);
    \coordinate (w1) at (0,-1);
    \coordinate (w2) at (0,0);
    \coordinate (w3) at (0,1);
    \coordinate (w4) at (0,2);

    \draw[edge] (v0) -- (w0) -- (v4) -- (w4) -- cycle;
    \draw[edge] (v1) -- (w1) -- (v3) -- (w3) -- cycle;
    \draw[edge] (v0) -- (w1);
    \draw[edge] (v1) -- (w4);
    \draw[edge] (v3) -- (w0);
    \draw[edge] (v4) -- (w3);
    \draw[edge] (v0) -- (v4);
    \draw[edge] (w0) -- (w4);

    \foreach \point in {v0,v1,v2,v3,v4,w0,w1,w2,w3,w4}
    \fill[black] (\point) circle (2pt);
    
    \fill[white] (v2) circle (1pt);
  \end{tikzpicture}
  \caption{\label{fig:example:conjecture}
    The non-regular sum triangulation from \autoref{example:conjecture}.
  }
\end{figure}

\begin{example}\label{example:conjecture}
  Let $P$ and $Q$ be the same point configuration
    \[
      P = Q = \{-2,-1,0,1,2\},
    \]
  with the same triangulation
    \[
      \triangle_P = \triangle_Q = \langle [-2,-1],[-1,0],[0,1],[1,2] \rangle.
    \]
  Furthermore we define two compatible webs of stars via
  \begin{align*}
    \alpha : \triangle_P^{=1} &\to \cB_{\triangle_Q}  &  \beta : \triangle_Q^{=1} &\to \cB_{\triangle_P} \\
    [-2,-1] &\mapsto \langle [-1,0],[0,1],[1,2]\rangle_{\triangle_Q} &
       [-2,-1] &\mapsto \langle [-2,-1],[-1,0],[0,1] \rangle_{\triangle_P} \\
    [-1,0] &\mapsto \langle [-1,0],[0,1]\rangle_{\triangle_Q} &
       [-1,0] &\mapsto \emptyset \\
    [0,1] &\mapsto \langle [-1,0],[0,1]\rangle_{\triangle_Q} &
       [0,1] &\mapsto \emptyset \\
    [1,2] &\mapsto \langle [-2,-1],[-1,0],[0,1]\rangle_{\triangle_Q} &
       [1,2] &\mapsto \langle [-1,0],[0,1],[1,2] \rangle_{\triangle_P} \\
  \end{align*}
  Clearly they do not meet the condition in \autoref{conj:regular}, because $\alpha([-2,-1])$ and $\alpha([1,2])$ are not $\subseteq$-comparable.
  The corresponding sum-triangulation $\triangle_{P\oplus Q}$ (cf.~\autoref{fig:example:conjecture}) is not regular.

  In a more complex situation in higher dimension we might get something similar.
  Maybe one needs to intersect the simplicial complex with an appropriate lower dimensional subspace to prove non-regularity.
\end{example}

\appendix

\begin{figure}[htb]
  \centering
    \includegraphics{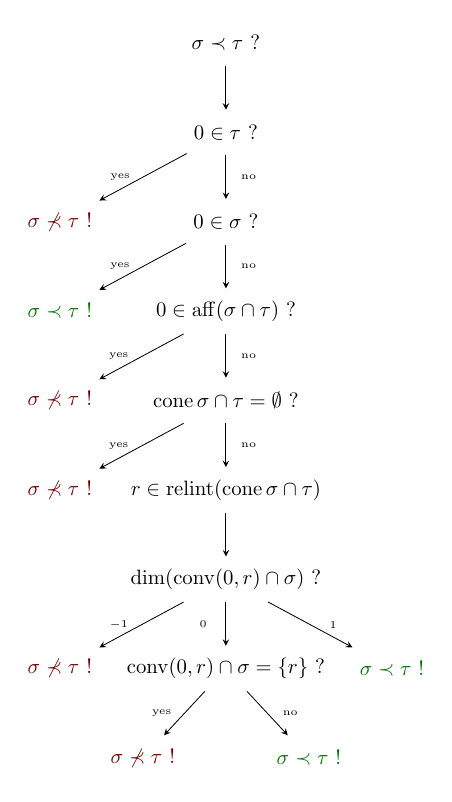}
  \caption{\label{fig:appendix:decision_tree}
    Decision diagram to decide whether or not $\sigma \prec \tau$.
  }
\end{figure}

\section{Deciding comparability \except{toc}{\\} in the stabbing order}
\phantomsection\label{sec:deciding_comparability}
\noindent
The flowchart in \autoref{fig:appendix:decision_tree} gives an algorithm to determine whether $\sigma \prec \tau$ holds for two simplices in a simplicial complex. We prove its correctness step by step:

\begin{enumerate}
\item
Cells containing the origin are $\preceq$-minimal, as they always lie in the same half space as the origin.

\item
If $0\in\sigma\cap\tau$, then $\sigma\not\prec\tau $ since every separating hyperplane is linear.

\item
If $0\in \aff(\sigma \cap \tau)$, then again every separating hyperplane is linear, since it must contain the affine hull of the intersection, and so $\sigma\not\prec\tau$.

\item
If $\cone \sigma \cap \tau =\emptyset$ there exists no stabbing ray, as the set of all rays that stab both~$\sigma$ and~$\tau$ is $\cone \sigma \cap \cone \tau$. But the existence of such a ray is a necessary condition according to \autoref{lemma:find_intersecting_ray}, and therefore $\sigma\not\prec\tau$.

\item
Let $r \in \relint(\cone \sigma \cap \tau)$. The dimension of the intersection of $\sigma$ and the line segment $\overline{0r}$ can either be $-1$, $0$ or $1$.

\begin{enumerate}
\item If $\dim\overline{0r}\cap\sigma=-1$, the intersection is empty.  But the ray spanned by~$r$ must   intersect $\sigma$, and $\overline{0 r} \cap \sigma=\emptyset$ implies that $\lambda r \in \sigma$   for some $\lambda > 1$.  Every separating hyperplane must separate $r$ from $\lambda r$, implying   that $\tau$ lies on the same half space as the origin, and therefore $\sigma$ does not precede   $\tau$.

\item If $\dim\overline{0r}\cap\sigma=1$, there exists $\lambda < 1$ with $\lambda r \in \sigma$. As every separating hyperplane must separate $r$ from $\lambda r$ we get that $\sigma$ always lies in the same half space as the origin.  And since $0\notin \aff(\sigma \cap \tau)$ we also find at least one non-linear hyperplane separating $\sigma$ and $\tau$; hence $\sigma \prec \tau$.

\item If $\overline{0r}\cap\sigma=\{ x \}$ and $r \neq x$, then $x = \lambda r$ for some $\lambda < 1$, and the same argument as above yields that $\sigma \prec \tau$.
\end{enumerate}

\item For the last step assume that $r = x$. Then $r\in\partial\sigma$ and $\sigma \cap \tau = \{ r \}$, and we can find a linear supporting hyperplane of $\sigma$ which separates~$\sigma$ and~$\tau$.
Small perturbations of that hyperplane produce a separating hyperplane with $\tau$ in the same half space as the origin, and hence $\sigma\not\prec\tau$.
\end{enumerate}


\begin{thebibliography}{10}

\bibitem{AssarfJoswigPaffenholz:2014}
B.~Assarf, M.~Joswig, and A.~Paffenholz.
\newblock Smooth {F}ano polytopes with many vertices.
\newblock {\em Discrete Comput. Geom.}, 52(2):153--194, 2014.

\bibitem{AssarfNill:1409.7303}
B.~Assarf and B.~Nill.
\newblock A bound for the splitting of smooth {F}ano polytopes with many
  vertices.
\newblock {\em J. Algebraic Combin.}, 43(1):153--172, 2016.

\bibitem{Batyrev2007}
V.~V. Batyrev.
\newblock On the classification of toric {F}ano {$4$}-folds.
\newblock {\em J. Math. Sci. (New York)}, 94(1):1021--1050, 1999.

\bibitem{Book_CoxLittleSchenck}
D.~A. Cox, J.~B. Little, and H.~K. Schenck.
\newblock {\em Toric varieties}, volume 124 of {\em Graduate Studies in
  Mathematics}.
\newblock American Mathematical Society, Providence, RI, 2011.

\bibitem{de2010triangulations}
J.~De~Loera, J.~Rambau, and F.~Santos.
\newblock {\em Triangulations: Structures for Algorithms and Applications}.
\newblock Algorithms and Computation in Mathematics. Springer-Verlag, 2010.

\bibitem{DMV:polymake}
E.~Gawrilow and M.~Joswig.
\newblock \texttt{polymake}: a framework for analyzing convex polytopes.
\newblock In {\em Polytopes---combinatorics and computation (Oberwolfach,
  1997)}, volume~29 of {\em DMV Sem.}, pages 43--73. Birkh\"auser, Basel, 2000.

\bibitem{HerrmannJoswig:2010}
S.~Herrmann and M.~Joswig.
\newblock Totally splittable polytopes.
\newblock {\em Discrete Comput. Geom.}, 44(1):149--166, 2010.

\bibitem{Hudson.pl}
J.~F.~P. Hudson.
\newblock {\em Piecewise linear topology}.
\newblock Mathematics lecture note series. W.A. Benjamin, 1 edition, 1969.

\bibitem{KN5}
M.~Kreuzer and B.~Nill.
\newblock Classification of toric {F}ano 5-folds.
\newblock {\em Adv. Geom.}, 9(1):85--97, 2009.

\bibitem{McMullen1976}
P.~McMullen.
\newblock Constructions for projectively unique polytopes.
\newblock {\em Discrete Mathematics}, 14(4):347--358, 1976.

\bibitem{OebroPhD}
M.~{\O}bro.
\newblock {\em Classification of smooth {Fano} polytopes}.
\newblock PhD thesis, University of Aarhus, 2007.
\newblock available at
  \url{https://pure.au.dk/portal/files/41742384/imf_phd_2008_moe.pdf}.

\bibitem{Paffenholz:1711.02936}
A.~Paffenholz.
\newblock \texttt{polyDB}: A database for polytopes and related objects, 2017.
\newblock Preprint \texttt{arXiv:1711.02936}.

\bibitem{PfeifleRambau:2003}
J.~Pfeifle and J.~Rambau.
\newblock Computing triangulations using oriented matroids.
\newblock In {\em Algebra, geometry, and software systems}, pages 49--75.
  Springer, Berlin, 2003.

\bibitem{TOPCOM}
J.~Rambau.
\newblock \texttt{TOPCOM}, version~0.17.5.
\newblock Available at \url{http://www.rambau.wm.uni-bayreuth.de/TOPCOM/},
  2015.

\end{thebibliography}
\end{document}